\documentclass{amsart}
\usepackage{latexsym}
\usepackage{amsmath}
\usepackage{amssymb}
\usepackage[all]{xy}

\newtheorem{thm}{Theorem}[section]
\newtheorem{prop}[thm]{Proposition}

\newtheorem{cor}[thm]{Corollary}
\newtheorem{lem}[thm]{Lemma}
\theoremstyle{definition}
\newtheorem{defn}[thm]{Definition}

\theoremstyle{remark}
\newtheorem{rem}[thm]{Remark}
\newtheorem{ex}[thm]{Example}

\newcommand{\K}{{\mathbb K}}
\newcommand{\Z}{{\mathbb Z}}
\newcommand{\Q}{{\mathbb Q}}
\newcommand{\D}{\text{D}}
\newcommand{\DD}{\text{\em D}}
\newcommand{\mapright}[1]{%
 \smash{\mathop{%
  \hbox to 1cm{\rightarrowfill}}\limits_{#1}}}
\newcommand{\maprightd}[2]{%
 \smash{\mathop{%
  \hbox to 1.2cm{\rightarrowfill}}\limits^{#1}\limits_{#2}}}
\newcommand{\mapleft}[1]{%
 \smash{\mathop{%
  \hbox to 1cm{\leftarrowfill}}\limits_{#1}}}
\newcommand{\mapleftu}[1]{%
 \smash{\mathop{%
  \hbox to 0.8cm{\leftarrowfill}}\limits^{#1}}}
\newcommand{\maprightu}[1]{%
 \smash{\mathop{%
  \hbox to 1cm{\rightarrowfill}}\limits^{#1}}}
\newcommand{\maprightud}[2]{%
 \smash{\mathop{%
  \hbox to 1cm{\rightarrowfill}}\limits^{#1}_{#2}}}
\newcommand{\mapleftud}[2]{%
 \smash{\mathop{%
  \hbox to 1cm{\leftarrowfill}}\limits^{#1}_{#2}}}

\newcounter{eqn}[section]

\def\theeqn{\textnormal{(\thesection.\arabic{eqn})}}

\def\eqnlabel#1{%
  \refstepcounter{eqn}%
  \label{#1}%
  \leqno{\theeqn}}

\begin{document}

\title[Derived string topology]{
Derived string topology and the Eilenberg-Moore spectral sequence  
}

\footnote[0]{


Department of Mathematical Sciences, 
Faculty of Science,  
Shinshu University,   
Matsumoto, Nagano 390-8621, Japan   
e-mail:{\tt kuri@math.shinshu-u.ac.jp} 

D\'epartement de Math\'ematiques
Facult\'e des Sciences,
Universit\'e d'Angers,
49045 Angers,
France 
e-mail:{\tt luc.menichi@univ-angers.fr}

Department of Mathematical Sciences, 
Faculty of Science,  
Shinshu University,   
Matsumoto, Nagano 390-8621, Japan   
e-mail:{\tt naito@math.shinshu-u.ac.jp} 
}

\author{Katsuhiko KURIBAYASHI, Luc MENICHI and Takahito Naito}
   
\maketitle


\begin{abstract}
Let $M$ be any simply-connected Gorenstein space over any field.
F\'elix and Thomas have extended to simply-connected Gorenstein spaces, the loop (co)products of Chas and Sullivan
on the homology of the free loop space $H_*(LM)$.
We describe these loop (co)products in terms  of 
the torsion and extension functors by developing string topology in appropriate derived categories. 
As a consequence, we show that the Eilenberg-Moore spectral sequence converging to the loop homology of 
a Gorenstein space admits a multiplication 
and a comultiplication with shifted degree which are compatible with the loop product and the loop coproduct of its target, respectively.

We also define a generalized cup product on the Hochschild cohomology $HH^*(A,A^\vee)$ of a commutative
Gorenstein algebra $A$ and show that over $\mathbb{Q}$, $HH^*(A_{PL}(M),A_{PL}(M)^\vee)$ is isomorphic
as algebras to $H_*(LM)$. Thus, when $M$ is a Poincar\'e duality space, we recover the
isomorphism of algebras $\mathbb{H}_*(LM;\mathbb{Q})\cong HH^*(A_{PL}(M),A_{PL}(M))$ of F\'elix and Thomas.
\end{abstract}

\section{Introduction}
There are several spectral sequences concerning main players in string topology 
\cite{C-J-Y, C-L, LeB, S, Kuri2011}. Cohen, Jones and Yan \cite{C-J-Y} 
have constructed a loop algebra spectral sequence which is of the Leray-Serre type. 
The Moore spectral sequence converging to the Hochschild cohomology ring 
of a differential graded algebra  is endowed with an algebra structure \cite{F-T-VP}  
and moreover a Batalin-Vilkovisky algebra structure \cite{Kuri2011}, 
which are compatible with such a structure of the target.  
Very recently, Shamir \cite{S} has constructed a Leray-Serre type spectral sequence converging to 
the Hochschild cohomology ring of a differential graded algebra. 
Then as announced by McClure~\cite[Theorem B]{McClure}, one might expect that the Eilenberg-Moore 
spectral sequence (EMSS), which converges to the loop homology of a closed oriented manifold and 
of a more general Gorenstein space,   
enjoys a multiplicative structure  corresponding to the loop product.

The class of Gorenstein spaces 
contains Poincar\'e duality spaces, for example closed oriented manifolds, and Borel constructions, 
in particular, the classifying spaces of connected Lie groups; see \cite{FHT_G, Murillo, KMnoetherian}. 
In \cite{F-T}, F\'elix and Thomas develop string topology on Gorenstein spaces.
As seen in string topology, the shriek map (the wrong way map) plays an important role when defining string operations. 
Such a map for a Gorenstein space appears in an appropriate derived category.  
Thus we can discuss string topology due to Chas and Sullivan in the more general setting with  
cofibrant replacements of the singular cochains on spaces. 

In the remainder of this section, our main results are surveyed. 
We describe explicitly the loop (co)products for a Gorenstein space in terms of 
the differential torsion product and the extension functors; see Theorems  \ref{thm:torsion_product}, 
\ref{thm:torsion_coproduct} and \ref{thm:freeloopExt}. 
The key idea of the consideration comes from the general setting 
in \cite{F-T} for defining string operations mentioned above. 
Thus our description of the loop (co)product with 
derived functors fits {\it derived string topology}, namely the framework of string topology due to F\'elix and Thomas. 
Indeed, according to expectation, the full descriptions of the products with derived functors 
permits us to give the EMSS (co)multiplicative structures  
which are compatible with the dual to the loop (co)products  of its target; see Theorem \ref{thm:EMSS}.   

By dualizing the EMSS,  we obtain a new spectral sequence converging to 
the Chas-Sullivan relative loop homology algebra with coefficients in a field $\K$ of a Gorenstein space $N$ over a space $M$. 
We observe that the $E_2$-term of the dual EMSS is represented by the Hochschild cohomology ring of $H^*(M; \K)$ with coefficients in 
the shifted homology of $N$; 
see Theorems \ref{thm:loop_homology_ss}. 
It is conjectured that there is an isomorphism of graded algebras between the loop homology of $M$
and the Hochschild cohomology of the singular cochains on $M$. But over $\mathbb{F}_p$,
even in the case of a simply-connected closed orientable manifold, there is no complete written proof
of such an isomorphism of algebras (See~\cite[p. 237]{F-T-VP} for details).
Anyway, even if we assume such isomorphism, it is not clear that the spectral sequence
obtained by filtering Hochschild cohomology is isomorphic to the dual EMSS
although these two spectral sequences have the same $E_2$ and $E_\infty$-term.
It is worth stressing that the EMSS in Theorem \ref{thm:EMSS} is applicable to each space in the more wide class of Gorenstein spaces and 
is moreover endowed with both the loop product and the loop coproduct.  
Let $N$ be a simply-connected space whose cohomology 
is of finite dimension and is generated by a single element. Then 
explicit calculations of the dual EMSS made in the sequel \cite{K-M-N2} to this paper 
yield that  the loop homology of $N$ is isomorphic 
to the Hochschild cohomology of $H^*(N; \K)$ as an algebra. 
This illustrates computability of our spectral sequence 
in Theorem \ref{thm:loop_homology_ss}.

With the aid of the torsion functor descriptions of the loop (co)products, we see that 
the composite 
$(\text{\it the loop product})\circ(\text{\it the loop coproduct})$ is trivial for a 
simply-connected Poincar\'e duality space; see Theorem \ref{thm:loop(co)product}. 
Therefore, the same argument as in the proof of \cite[Theorem A]{T} deduces that 
if string operations on a Poincar\'e duality space 
gives rise to a 2-dimensional TQFT, then all operations associated to surfaces of genus 
at least one vanish. For a more general Gorenstein space, an obstruction for 
the composite to be trivial can be found in a hom-set, namely 
the extension functor, in an appropriate derived category; see Remark \ref{rem:obstruction}.   
This small but significant result also asserts an advantage of derived string topology. 
 
It is also important to mention that in the Appendices, 
we have paid attention to signs and extended the properties of shriek maps on Gorenstein spaces
given in \cite{F-T}, in order to prove that the loop product is associative and commutative for
Poincar\'e duality space.


\section{Derived string topology and main results} 
The goal of this section is to state our results in detail. The proofs are found in Sections 3 to 7.   

We begin by recalling the most prominent result on shriek maps due to F\'elix and Thomas, 
which supplies string topology with many homological and homotopical algebraic tools. 
Let $\K$ be a field of arbitrary characteristic. 
In what follows, we denote by $C^*(M)$ and $H^*(M)$ 
the normalized singular cochain algebra of a space $M$ with coefficients in $\K$ 
and its cohomology, respectively. 
For a differential graded algebra $A$, let $\D(\text{Mod-}A)$ and $\D(A\text{-Mod})$ be 
the derived categories of right $A$-modules and left $A$-modules, respectively. 
Unless otherwise explicitly stated, it is assumed that a space has the homotopy type 
of a CW-complex whose homology with coefficients in an underlying field is of finite type.

Consider a pull-back diagram ${\mathcal F}$: 
$$
\xymatrix@C30pt@R15pt{ 
X \ar[r]^{g} \ar[d]_q & E \ar[d]^p \\
N \ar[r]_f & M}
$$ 
in which $p$ is a fibration over a simply-connected Poincar\'e duality space $M$ of dimension $m$ with 
the fundamental class $\omega_M$ and $N$ is 
a Poincar\'e duality space of dimension $n$ with the fundamental class $\omega_N$. 

\begin{thm}\label{thm:main_F-T}{\em (}\cite{L-SaskedbyFelix},\cite[Theorems 1 and 2]{F-T}{\em )} 
With the notation above there exist unique elements
$$
f^! \in \text{\em Ext}^{m-n}_{C^*(M)}(C^*(N), C^*(M)) \ \ \text{and} \  \ g^! \in \text{\em Ext}^{m-n}_{C^*(E)}(C^*(X), C^*(E))
$$
such that $H^*(f^!)(\omega_N)=\omega_M$ and in $\DD(\text{\em Mod-}C^*(M))$, the following diagram is commutative 
$$
\xymatrix@C25pt@R15pt{ 
C^*(X) \ar[r]^(0.4){g^!} & C^{*+m-n}(E)  \\
C^*(N) \ar[r]_(0.4){f^!} \ar[u]^{q^*}& C^{*+m-n}(M).  \ar[u]_{p^*}
}
$$ 
\end{thm}

Let $A$ be a differential graded augmented algebra over $\K$.  We call $A$ a {\it Gorenstein algebra} of dimension $m$ 
if 
$$
\dim \text{Ext}_A^*(\K, A) = 
\left\{
\begin{array}{l}
0 \ \ \text{if} \ *\neq m,  \\
1  \ \ \text{if} \  *= m. 
\end{array}
\right.
$$ 
A path-connected space $M$ is called a $\K$-{\it Gorenstein space}  (simply, Gorenstein space) 
of dimension $m$ if the normalized singular cochain 
algebra $C^*(M)$ with coefficients in $\K$  is a Gorenstein algebra of dimension $m$. 
We write $\dim M$ for the dimension $m$.  

The result \cite[Theorem 3.1]{FHT_G} yields that a simply-connected Poincar\'e duality space, for example 
a simply-connected closed orientable manifold, is Gorenstein.  
The classifying space $BG$ of connected Lie group $G$ and the Borel construction $EG\times_GM$ for 
a simply-connected Gorenstein space $M$ with $\dim H^*(M; \K)< \infty$ on which $G$ acts  
are also examples of Gorenstein spaces; see \cite{FHT_G, Murillo, KMnoetherian}.  
Observe that, for a closed oriented manifold $M$, 
$\dim M$ coincides with the ordinary dimension of $M$ and that for the classifying space $BG$ of a connected Lie group, 
$\dim BG= -\dim G$. Thus the dimensions of Gorenstein spaces may become negative.    

The following theorem enables us to generalize the above result concerning shriek maps 
on a Poincar\'e duality space to that on a Gorenstein space. 

\begin{thm}{\em(}\cite[Theorem 12]{F-T}{\em)}\label{thm:ext}
Let $X$ be a simply-connected $\K$-Gorenstein space of dimension $m$ whose cohomology 
with coefficients in  $\K$ is of finite type. 
Then 
$$
\text{\em Ext}^*_{C^*(X^n)}(C^*(X), C^*(X^n)) \cong H^{*-(n-1)m}(X), 
$$
where $C^*(X)$ is considered a $C^*(X^n)$-module via the diagonal map $\Delta : X \to X^n$.   
\end{thm}

We denote by $\Delta^!$ the map in $\D(\text{Mod-}C^*(X^n))$ which 
corresponds to a generator of 
$\text{Ext}^{(n-1)m}_{C^*(X^n)}(C^*(X), X^*(X^n)) \cong H^0(X)$. 
Then, for a Gorenstein space $X$ of dimension $m$ and a fibre square 
$$
\xymatrix@C25pt@R15pt{ 
E' \ar[r]^{g} \ar[d]_{p'} & E \ar[d]^p \\
X^{} \ar[r]_\Delta & X^{n} , } 
$$ 
there exists a unique map $g^!$ in $\text{Ext}_{C^*(E)}^{(n-1)m}(C^*(E'), C^*(E))$ which fits into 
the commutative diagram in $\D(\text{Mod-}C^*(X^n))$
$$
\xymatrix@C25pt@R15pt{ 
C^*(E') \ar[r]^{g^!} & C^*(E)  \\
C^*(X^{}) \ar[r]_{\Delta^!} \ar[u]^{(p')^*}& C^*(X^{n}) . \ar[u]_{p^*}
}
$$ 
We remark that the result follows from the same proof as that of Theorem \ref{thm:main_F-T}. 

Let $\xymatrix@C15pt@R15pt{K & \ar[l]_{f} A \ar[r]^{g} & L}$  be a diagram in the category of  
differential graded algebras (henceforth called DGA's). We consider $K$ and $L$ 
right and left modules over $A$ via maps $f$ and $g$, respectively.  
Then the differential torsion product $\text{Tor}_A(K, L)$ is 
denoted by $\text{Tor}_A(K, L)_{f, g}$ when the actions are emphasized.

We recall here the Eilenberg-Moore map. 
Consider the pull-back diagram ${\mathcal F}$ mentioned above, in which 
$p$ is a fibration and $M$ is a simply-connected space. Let 
$\varepsilon : F \to C^*(E)$ be a left semi-free resolution of $C^*(E)$ in 
$C^*(M)\text{-Mod}$ the category of left $C^*(M)$-modules. Then the Eilenberg-Moore map 
$$
EM : \text{Tor}_{C^*(M)}^*(C^*(N), C^*(E))_{f^*, p^*} =H(C^*(N)\otimes_{C^*(M)}F) 
\longrightarrow H^*(X)
$$
is defined by 
$
EM(x\otimes_{C^*(M)} u) = q^*(x)\smile (g^*\varepsilon(u))
$ 
for $x\otimes_{C^*(M)} u \in C^*(N)\otimes_{C^*(M)}F$. 
Observe that in the same way, we can define the Eilenberg-Moore map 
by using a semi-free resolution of $C^*(N)$ as a right $C^*(M)$-module.
We see that the map $EM$ is an isomorphism of graded algebras with respect to the cup
products; see \cite{G-M} for example.   
In particular, for a simply-connected space $M$,
consider the commutative diagram,
$$
\xymatrix@C30pt@R15pt{ 
LM  \ar[r]^{} \ar[d]_{ev_0} & M^I \ar[d]_{p=(ev_0,ev_1)}
& M\ar[l]_{\sigma}^\simeq\ar[ld]^-{\Delta}\\
M \ar[r]_-\Delta & M\times M} 
$$ 
where $ev_i$ stands for the evaluation map at $i$
and $\sigma:M\buildrel{\simeq}\over\hookrightarrow M^I$ for the inclusion
of the constant paths.
We then obtain the composite $EM' :$
$$
\xymatrix@1{
H^*(LM) &
\text{Tor}_{C^*(M^{\times 2})}^*(C^*M, C^*M^I)_{\Delta^*,p^*}\ar[l]_-{EM}^-\cong \ar[r]^{\text{Tor}_{1}(1,\sigma^*)}_\cong
& \text{Tor}_{C^*(M^{\times 2})}^*(C^*M, C^*M)_{\Delta^*,\Delta^*}
}.
$$
Our first result states that the torsion functor $ \text{Tor}_{C^*(M^{\times 2)}}^*(C^*(M), C^*(M))_{\Delta^*,\Delta^*}$
admits (co)products which are compatible with $EM'$.

In order to describe such a result, 
we first recall the definition of the loop product on a simply-connected Gorenstein space. Consider the diagram 
$$
\xymatrix@C25pt@R15pt{ 
LM \ar[d]_{ev_0} & LM\times_M LM \ar[l]_(0.6){Comp} \ar[d] \ar[r]^q & LM \times LM \ar[d]^{(ev_0,ev_1)} \\
M \ar@{=}[r] & M \ar[r]_{\Delta} & M\times M, } 
\eqnlabel{add-0}
$$ 
where the right-hand square is the pull-back of the diagonal map $\Delta$, $q$ is the inclusion and 
$Comp$ denotes the concatenation of loops.  
By definition the composite 
$$
q^!\circ (Comp)^* : C^*(LM) \to C^*(LM\times_M LM) \to C^*(LM \times LM)
$$ 
induces the dual to the loop product $Dlp$ on $H^*(LM)$; see \cite[Introduction]{F-T}. 
We see that $C^*(LM)$ and $C^*(LM\times LM)$ are $C^*(M\times M)$-modules via the map 
 $ev_0\circ \Delta$ and $(ev_0, ev_1)$, respectively. Moreover since $q^!$ is a morphism of  $C^*(M\times M)$-modules, it follows that 
so is $q^!\circ (Comp)^*$.  
The proof of Theorem \ref{thm:main_F-T} states that the map $q^!$ is obtained extending the shriek map $\Delta^!$, which is first given,   
in the derived category $\D(\text{Mod-}C^*(M\times M))$. This fact allows us to formulate  
$q^!$ in terms of differential  torsion functors. 

\begin{thm} \label{thm:torsion_product}
Let $M$ be a simply-connected Gorenstein space of dimension $m$.
Consider the comultiplication $\widetilde{(Dlp)}$ given by the composite
$$
{\footnotesize
\xymatrix@C8pt@R15pt{ 
\text{\em Tor}^*_{C^*(M^{2})}(C^*(M), C^*(M))_{\Delta^*, \Delta^*}
\ar[r]^-{\text{\em Tor}_{p_{13}^*}(1, 1)}  \ar@{.>}[dd]                  & 
             \text{\em Tor}^*_{C^*(M^{3})}(C^*(M), C^*(M))_{((1\times \Delta)\circ \Delta)^*, ((1\times \Delta)\circ \Delta)^*}\\
  &               \text{\em Tor}^*_{C^*(M^{4})}(C^*(M), C^*(M^2))_{(\Delta^2\circ\Delta)^*, {\Delta^2}^*} 
                             \ar[u]_{\text{\em Tor}_{(1\times \Delta \times 1)^*}(1, {\Delta}^*)}^{\cong} 
                                          \ar[d]^{\text{\em Tor}_{1}(\Delta^!, 1)} \\
\left(\text{\em Tor}^*_{C^*(M^{2})}(C^*(M), C^*(M))_{\Delta^*, \Delta^*}^{\otimes 2}\right)^{*+m}\ar[r]^-{\cong}_{\widetilde{\top}} & 
\text{\em Tor}^{*+m}_{C^*(M^{4})}(C^*(M^{2}), C^*(M^2))_{{\Delta^2}^*, {\Delta^2}^*}. 
}
}
$$
See Remark~\ref{definition generalized T-product} below for the definition of $\widetilde{\top}$.
Then the composite $EM' :$
$$
\xymatrix@C15pt@R15pt{ 
H^*(LM)\ar[r]^-{EM^{-1}}_-\cong
& \text{\em Tor}_{C^*(M^{2})}^*(C^*(M), C^*(M^I))_{\Delta^*,p^*} \ar[r]^{\text{\em Tor}_{1}(1,\sigma^*)}_\cong
&\text{\em Tor}_{C^*(M^{2})}^*(C^*(M), C^*(M))_{\Delta^*,\Delta^*}
}$$
is an isomorphism which respects the dual to the loop product $Dlp$
and the comultiplication $\widetilde{(Dlp)}$ defined here.
\end{thm}

\begin{rem}\label{definition generalized T-product}
The isomorphism $\widetilde{\top}$ in
Theorem~\ref{thm:torsion_product}
is the canonical map defined by~\cite[p. 26]{G-M}  or
by~\cite[p. 255]{Mccleary} as the composite
$$
{\footnotesize 
\xymatrix@C60pt@R20pt{ 
 \text{Tor}^*_{C^*(M^{2})}(C^*(M), C^*(M))^{\otimes 2}
\ar[r]^-\top
& \text{Tor}^*_{C^*(M^{2})^{\otimes 2}}(C^*(M)^{\otimes 2},
C^*(M)^{\otimes 2})\ar[d]^{\text{Tor}_\gamma(\gamma,\gamma)}\\
\text{Tor}^*_{C^*(M^{4})}(C^*(M^{2}), C^*(M^2))
\ar[r]_-{\text{Tor}_{EZ^\vee}(EZ^\vee,EZ^\vee)}^-\cong
&\text{Tor}^*_{(C_*(M^{2})^{\otimes 2})^\vee}((C_*(M)^{\otimes 2})^\vee,
(C_*(M)^{\otimes 2})^\vee)
}
}
$$
where $\top$ is the $\top$-product of
Cartan-Eilenberg~\cite[XI. Proposition 1.2.1]{CartanEilenberg}
or~\cite[VIII.Theorem 2.1]{MacLanehomology}, $EZ:C_*(M)^{\otimes
  2}\buildrel{\simeq}\over\rightarrow C_*(M^2)$ denotes the Eilenberg-Zilber
quasi-isomorphism
 and $\gamma:\text{Hom}(C_*(M), \K)^{\otimes 2}\rightarrow
 \text{Hom}(C_*(M)^{\otimes 2}, \K)$ is the canonical map.
\end{rem}
It is worth mentioning that this theorem gives an intriguing decomposition of 
the cup product on the Hochschild cohomology of 
a commutative algebra; see Lemma \ref{lem:generalizeddecompositioncup} below.

The loop coproduct on a Gorenstein space is also interpreted in terms of torsion products. 
In order to recall the loop coproduct, we consider the commutative diagram 
$$
\xymatrix@C25pt@R20pt{ 
LM \times LM  & LM\times_M LM \ar[l]_(0.5){q} \ar[d] \ar[r]^(0.6){Comp} & LM \ar[d]^l \\
& M \ar[r]_{\Delta} & M\times M, } 
\eqnlabel{add-2}
$$ 
where $l : LM \to M\times M$ is a map defined by $l(\gamma)= (\gamma(0), \gamma(\frac{1}{2}))$. 
By definition,  the composite 
$$
Comp^!\circ q^* : C^*(LM\times LM ) \to C^*(LM\times_M LM) 
\to C^*(LM)
$$ 
induces the dual to the loop coproduct $Dlcop$ on $H^*(LM)$. 

Note that we apply Theorem \ref{thm:main_F-T} to (2.2) in defining the loop coproduct. 
On the other hand, applying Theorem \ref{thm:main_F-T} to the diagram (2.1), the loop product is defined.

\begin{thm} \label{thm:torsion_coproduct}
 Let $M$ be a simply-connected Gorenstein space of dimension $m$.
Consider the multiplication defined by the composite 
$$
{\footnotesize
\xymatrix@C20pt@R15pt{ 
\left(\text{\em Tor}^{*}_{C^*(M^{2})}(C^*(M), C^*(M))_{\Delta^*, \Delta^*}\right)^{\otimes 2} \ar@{.>}[dd] \ar[r]^\cong_{\widetilde{\top}}
&\text{\em Tor}^*_{C^*(M^4)}(C^*(M^2), C^*(M^2))_{{\Delta^2}^* \!, {\Delta^2}^*}\ar[d]^{\text{\em Tor}_{1}(\Delta^*, 1)} \\
& \text{\em Tor}^*_{C^*(M^{4})}(C^*(M), C^*(M^2))_{(\Delta^2\circ\Delta)^*, {\Delta^2}^*}\ar[d]^{\text{\em Tor}_{1}(\Delta^!, 1)}  \\
\text{\em Tor}^{*+m}_{C^*(M^{2})}(C^*(M), C^*(M))_{\Delta^*, \Delta^*}
&\text{\em Tor}^{*+m}_{C^*(M^{4})}(C^*(M^2), C^*(M^2))_{{\gamma'}^*, {\Delta^2}^*}\ar[l]^{\text{\em Tor}_{\alpha^*}(\Delta^*, \Delta^*)}_\cong
}
}
$$
where the maps $\alpha : M^{2} \to M^{4}$ and 
$\gamma' : M^{2} \to M^{4}$ are defined by $\alpha(x, y) = (x, y, y, y)$ and 
$\gamma'(x, y) =(x, y, y, x)$. See remark~\ref{definition generalized T-product} above for the definition of $\widetilde{\top}$.
Then the composite $EM' :$
$$
\xymatrix@C15pt@R20pt{ 
H^*(LM)\ar[r]^-{EM^{-1}}_-\cong
& \text{\em Tor}_{C^*(M^{2})}^*(C^*(M), C^*(M^I))_{\Delta^*,p^*} \ar[r]^{\text{\em Tor}_{1}(1,\sigma^*)}_\cong
&\text{\em Tor}_{C^*(M^{2})}^*(C^*(M), C^*(M))_{\Delta^*,\Delta^*}
}
$$
is an isomorphism respects the dual to the loop coproduct $Dlcop$ and the multiplication defined here.
\end{thm}

\begin{rem}\label{rem:relative_cases}
A relative version of the loop product is also in our interest. Let $f : N \to M$ be a map. 
Then by definition, the relative loop space $L_fM$ fits into the pull-back diagram 
$$
\xymatrix@C30pt@R20pt{ 
L_fM \ar[r]^{} \ar[d] & M^I \ar[d]^{(ev_0,  ev_1)} \\
N \ar[r]_(0.4){(f, f)} & M\times M , 
}
$$
where $ev_t$ denotes the evaluation map at $t$. We may write $L_NM$ for the relative loop space $L_fM$ 
in case there is no danger of confusion. 
Suppose further that  $M$ is simply-connected and has a base point.
Let $N$ be a simply-connected Gorenstein space. Then the diagram
$$
\xymatrix@C25pt@R20pt{ 
L_NM & L_NM\times_N L_NM \ar[l]_(0.6){Comp} \ar[d] \ar[r]^q & L_NM \times L_NM \ar[d]^{(ev_0,ev_1)} \\
& N \ar[r]_{\Delta} & N\times N} 
$$ 
gives rise to the composite 
$$
q^!\circ (Comp)^* : C^*(L_NM) \to C^*(L_NM\times_N L_NM) \to C^*(L_NM \times L_NM)
$$ 
which, by definition, induces the dual to 
the relative loop product $Drlp$ on the cohomology $H^*(L_NM)$ with degree $\dim N$; see \cite{F-T-VP,G-S} 
for case that $N$ is a smooth manifold.
Since the diagram above corresponds to the diagram (2.1), 
the proof of Theorem \ref{thm:torsion_product} permits one to conclude that $Drlp$ has also the same description as in 
Theorem \ref{thm:torsion_product}, where $C^*(N)$ is put instead of $C^*(M)$ in the left-hand variables of the 
torsion functors in the theorem. 

As for the loop coproduct, we cannot define its relative version in natural way because of the evaluation map $l$  of loops at $\frac{1}{2}$; see 
the diagram (2.2). Indeed the point $\gamma(\frac{1}{2})$ for a loop $\gamma$ in $L_NM$ is not necessarily in $N$. 
\end{rem}

The associativity of $Dlp$ and $Dlcop$ on a Gorenstein space is an important issue. 
We describe here an algebra structure on the shifted homology 
$H_{-*+d}(L_NM)=(H^*(L_NM)^\vee)^{*-d}$ 
of a simply-connected Poincar\'e duality space $N$ of dimension $d$ with a map $f : N \to M$ to a simply-connected space. 

We define a map $m : H_*(L_NM)\otimes H_*(L_NM) \to H_*(L_NM)$ of degree $d$ by  
$$
m(a\otimes b) = (-1)^{d(|a|+d)}((Drlp)^\vee)(a\otimes b)
$$
for $a$ and $b \in H_*(L_NM)$;
see \cite[sign of Proposition 4]{C-J-Y} or~\cite[Definition
3.2]{Tamanoi:cap products}. 
Moreover, put ${\mathbb H}_{*}(L_NM)= H_{*+d}(L_NM)$. Then we establish the following proposition. 

\begin{prop}\label{prop:loop_homology} 
Let $N$ be a simply-connected Poincar\'e duality space. Then 
the shifted homology ${\mathbb H}_{*}(L_NM)$ 
is an associative algebra with respect to the product $m$. Moreover, if $M=N$, then the shifted homology 
${\mathbb H}_{*}(LM)$ is graded commutative. 
\end{prop}

As mentioned below, the loop product on $L_NM$ is not commutative in general. 

We call a bigraded vector space $V$ a {\it bimagma} with shifted degree $(i, j)$ if $V$ 
is endowed with a multiplication 
$V\otimes V \to V$ and a comultiplication $V \to V\otimes V$ of degree $(i, j)$. 

Let $K$ and $L$ be objects in $\text{Mod-}A$ and $A\text{-Mod}$, respectively. 
Consider a torsion product of the form $\text{Tor}_{A}(K, L)$ which 
is the homology of the derived tensor product $K\otimes_A^{{\mathbb L}}L$.  
The external degree of the bar resolution of the second variable $L$  
filters the torsion products. Indeed, we can regard 
the torsion product $\text{Tor}_{A}(K, L)$ as the homology 
$H(M\otimes_AB(A, A, L))$ with the bar resolution $B(A, A, L)\to L$ of $L$. 
Then the filtration 
${\mathcal F}=\{F^p\text{Tor}_{A}(K, L)\}_{p\leq 0}$ of the torsion product  is defined by 
$$
F^p\text{Tor}_{A}(K, L) = \text{Im} \{i^* : 
H(M\otimes_AB^{\leq p}(A, A, L)) \to \text{Tor}_{A}(K, L)\}. 
$$ 
Thus the filtration 
${\mathcal F}=\{F^p\text{Tor}_{C^*(M^2)}(C^*(M), C^*(M^I))\}_{p\leq 0}$ 
induces a filtration of $H^*(LM)$ via 
the Eilenberg-Moore map for a simply-connected space $M$.

By adapting differential torsion functor descriptions of the loop (co)products in 
Theorems \ref{thm:torsion_product} and \ref{thm:torsion_coproduct}, we can give the EMSS 
a bimagma structure.

\begin{thm} \label{thm:EMSS}
Let $M$ be a simply-connected Gorenstein space of dimension $d$. 
Then the Eilenberg-Moore spectral sequence 
$\{E_r^{*,*}, d_r\}$ converging to $H^*(LM; \K)$ admits loop (co)products which 
is compatible with those in the target; that is, each term $E_r^{*,*}$ is endowed with a comultiplication 
$Dlp_r : E_r^{p,q} \to \oplus_{s+s' =p, t+t'=q+d}E_r^{s, t}\otimes E_r^{s',t'}$ and a multiplication 
$Dlcop_r : E_r^{s,t}\otimes E_r^{s', t'} \to E_r^{s+s', t+t'+d}$
which are compatible with differentials in the sense that  
$$Dlp_r d_r = (-1)^d(d_r\otimes 1 + 1\otimes d_r)Dlp_r  \ \ \text{and} \ \  
Dlcop_r(d_r\otimes 1+ 1\otimes d_r)=(-1)^dd_rDlcop_r.  
$$
Here $(d_r\otimes 1 + 1\otimes d_r)(a\otimes b)$ means
$d_ra\otimes b +(-1)^{p+q}a\otimes d_rb$ if $a\in E_r^{p, q}$.
Note the unusual sign $(-1)^d$.
Moreover the $E_\infty$-term 
$E_\infty^{*,*}$ is isomorphic to 
$\text{\em Gr}H^*(LM; \K)$ as a  bimagma with shifted degree $(0, d)$.   

\end{thm}
If the dimension of the Gorenstein space is non-positive, unfortunately
the loop product and the loop coproduct in the EMSS are trivial
and the only information that  Theorem~\ref{thm:EMSS} gives 
is the following corollary.

\begin{cor}\label{cor:trivial coproduct in EMSS}
Let $M$ be a simply-connected Gorenstein space of dimension $d$.
Assume that $d$ is negative or that $d$ is null and $H^*(M)$ is not concentrated in degree $0$.
Consider the filtration given by the cohomological Eilenberg-Moore spectral sequence
converging to $H^*(LM; \K)$.
Then the dual to the loop product and that to the loop coproduct increase 
both the filtration degree of $H^*(LM)$ by at least one.
\end{cor}

\begin{rem}
a) Let $M$ be a simply-connected closed oriented manifold. We can choose a map 
$\Delta^! : C^*(M) \to C^*(M\times M)$ so that $H(\Delta^!)w_M =w_{M\times M}$; 
that is, $\Delta^!$ is the usual shriek map in the cochain level.   
Then the map $Dlp$ and $Dlcop$ coincide with 
the dual to the loop product and to the loop coproduct in the sense of 
Chas and Sullivan \cite{C-S}, Cohen and Godin \cite{C-G}, respectively. 
Indeed, this fact follows from the uniqueness of shriek map and the comments in three paragraphs in the end of 
\cite[p. 421]{F-T}. 
Thus the Eilenberg-Moore spectral sequence in Theorem \ref{thm:EMSS} converges to 
$H^*(LM; \K)$ as an algebra and a coalgebra. 

b) Let $M$ be the classifying space $BG$ of a connected Lie group $G$. 
Since the homotopy fibre of $\Delta : BG \to BG \times BG$ in (2.1) and (2.2) is homotopy equivalent to $G$,
 we can choose the shriek map $\Delta^!$ described in Theorems \ref{thm:torsion_coproduct} and 
\ref{thm:torsion_product} as the integration along the fibre.    
Thus $q^!$ also coincides with the integration along the fibre; see 
\cite[Theorems 6 and 13]{F-T}.  This yields that the bimagma structure in 
$\text{Gr}H^*(LBG; \K)$ is induced by the loop product and coproduct in the sense of 
Chataur and Menichi \cite{C-M}. 

c) Let $M$ be the Borel construction $EG\times_G X$ of a connected compact Lie group $G$
acting on a simply-connected closed oriented manifold $X$.
In~\cite{BGNX}, Behrend, Ginot, Noohi and Xu defined a loop product and a loop coproduct on
the homology $H_*(L\frak{X})$ of free loop of a stack $\frak{X}$. Their main example of stack is the quotient
stack $[X/G]$ associated to a connected compact Lie group $G$ acting smoothly on a closed oriented manifold $X$.
Although F\'elix and Thomas did not prove it, we believe that their loop (co)products for the Gorenstein space
$M=EG\times_G X$ coincide with the loop (co)products for the quotient stack $[X/G]$ of~\cite{BGNX}.
\end{rem}


The following theorem is the main result of this paper.

\begin{thm} 
\label{thm:loop_homology_ss} 
Let $N$ be a simply-connected Gorenstein space of dimension $d$. 
Let $f : N\rightarrow M$ be a continuous map to a simply-connected space $M$.
Then the Eilenberg-Moore spectral sequence is a right-half plane cohomological
spectral sequence $\{{\mathbb E}_r^{*,*}, d_r\}$ converging to the Chas-Sullivan loop homology 
${\mathbb H}_*(L_NM)$  as an algebra with 
$$
{\mathbb E}_2^{*,*} \cong HH^{*, *}(H^*(M); {\mathbb H}_*(N))
$$
as a bigraded algebra; that is, there exists a
decreasing filtration $\{F^p{\mathbb H}_*(L_N M)\}_{p\geq 0}$ of 
$({\mathbb H}_*(L_N M),m)$ such that 
${\mathbb E}_\infty^{*,*} \cong Gr^{*,*}{\mathbb H}_*(L_N M)$ as a bigraded algebra, where 
$$
Gr^{p,q}{\mathbb H}_*(L_N M) = F^p{\mathbb
  H}_{-(p+q)}(L_NM)/F^{p+1}{\mathbb H}_{-(p+q)}(L_N M).
$$ 
Here the product on the $\mathbb{E}_2$-term is the cup
product (See Definition~\ref{cup product Hochschild} (1)) induced by
$$
(-1)^d\overline{H(\Delta^{!})^\vee}:\mathbb{H}_*(N)\otimes_{H^{*}(M)} \mathbb{H}_*(N)\rightarrow \mathbb{H}_*(N).
$$
Suppose further that $N$ is a Poincar\'e duality space. Then the $\mathbb{E}_2$-term is isomorphic to  the Hochschild cohomology 
$HH^{*, *}(H^*(M); {H}^*(N))$ with the cup product as an algebra.  
\end{thm}

Taking $N$ to be the point, we obtain the following well-known corollary.

\begin{cor} {\em (}cf. \cite[Corollary 7.19]{Mccleary}{\em )}
Let $M$ be a pointed topological space.
Then the Eilenberg-Moore spectral sequence
$E_2^{*,*}=\text{\em Ext}^{*,*}_{H^*(M)}(\K, \K)$ converging to
$H_*(\Omega M)$ is a spectral sequence of algebras with respect to the
Pontryagin product.
\end{cor}

When $M=N$ is a closed manifold, Theorem 2.11 has been announced by McClure in~\cite[Theorem B]{McClure}. But the proof has not appeared.
Moreover, McClure claimed that when $M=N$, the Eilenberg-Moore spectral sequence is a spectral sequence
of BV-algebras. We have not yet been able to prove this very interesting claim.


We summarize here spectral sequences converging the loop homology and the Hochschild cohomology of the singular cochain on a space, 
which are mentioned at the beginning of the Introduction. 

\medskip
\noindent
\ \ \  {\small 
\begin{tabular}{|l|l|}
\hline
The homological Leray-Serre type & The cohomological Eilenberg-Moore type \\
\hline
$E_{-p,q}^2=H^p(M; H_q(\Omega M))$  
        &  $E^{p,q}_2=HH^{p,q}(H^*(M); H^*(M))$  \\ 
 \  \  \  \  \  \  \  \  $ \Rightarrow {\mathbb H}_{-p+q}(LM)$ as an algebra, 
   &  \  \  \  \  \  \  $\Rightarrow {\mathbb H}_{-p-q}(LM)$ as an algebra, \\ 
where $M$ is a simply-connected closed & 
                where $M$ is a simply-connected Poincar\'e \\ 
 oriented manifold; see \cite{C-J-Y}.       &   duality  space; see Theorem \ref{thm:loop_homology_ss}. \\ \hline
$E_{p,q}^2=H^{-p}(M)\otimes \text{Ext}_{C^*(M)}^{-q}(\K, \K)$  & 
                            $E^{p,q}_2=HH^{p,q}(H^*(M); H^*(M))$ \\ 
 \  \  \  \  \  \ $\Rightarrow HH^{-p-q}(C^*(M); C^*(M))$   &  
             \  \  \  \  \  \   $\Rightarrow HH^{p+q}(C^*(M); C^*(M))$  \\
as an algebra, where $M$ is a simply-  &  
as a B-V algebra, where  $M$ is a simply- \\  
connected space whose cohomology is  &  connected Poincar\'e duality space; see \cite{Kuri2011}. \\
 locally finite; see \cite{S}. & \\ \hline 
\end{tabular} 
}

\medskip 
\noindent
Observe that  each spectral sequence in the table above converges strongly to the target.   

It is important to remark that, for a fibration $N \to X \to M$ of closed orientable manifolds,   
Le Borgne \cite{LeB} 
has constructed a spectral sequence converging to the loop homology 
${\mathbb H}_*(LX)$ as an algebra with 
$E_2 \cong {\mathbb H}_*(LM)\otimes {\mathbb H}_*(LN)$ under an appropriate assumption; 
see also \cite{C-L} for applications of the spectral sequence. 
We refer the reader to \cite{Me} for spectral sequences concerning a generalized homology theory 
in string topology.

We focus on a global nature of the loop (co)product. 
Drawing on  the torsion functor description of the loop product and the loop coproduct 
mentioned in Theorems \ref{thm:torsion_product} and  \ref{thm:torsion_coproduct}, we have  the following result. 

\begin{thm} 
\label{thm:loop(co)product}
Let $M$ be a simply-connected Poincar\'e duality space. Then the composite 
$(\text{the loop product})\circ(\text{the loop coproduct})
$ is trivial.  
\end{thm}

When $M$ is a connected closed oriented manifold, the triviality of this composite
was first proved by Tamanoi~\cite[Theorem A]{T}.
Tamanoi has also shown that this composite is trivial when 
$M$ is the classifying space $BG$ of a connected Lie group $G$~\cite[Theorem 4.4]{Tamanoi:stabletrivial}.

We are aware that the description of the loop coproduct in 
Theorem \ref{thm:torsion_coproduct}  has no {\it opposite arrow} such as 
$\text{Tor}_{(1\times \Delta \times 1)^*}(1, \Delta^*)$ in Theorem \ref{thm:torsion_product}.  
This is a key to the proof of Theorem \ref{thm:loop(co)product}.
Though we have not yet obtained  
the same result  as Theorem \ref{thm:loop(co)product} on a more general Gorenstein space, some obstruction for the composite to be trivial is described in Remark \ref{rem:obstruction}. 


We may describe the loop product in terms of the extension functor. 

\begin{thm}\label{thm:freeloopExt}
Let $M$ be a simply-connected Poincar\'e duality space.
Consider the multiplication defined by the composite
$$
{\footnotesize 
\xymatrix@C20pt@R20pt{ 
\text{\em Ext}^*_{C^*(M^{2})}(C^*(M), C^*(M))_{\Delta^*, \Delta^*}^{\otimes 2}\ar[r]_-{\cong}^-{\widetilde{\vee}}\ar@{.>}[dd] 
&\text{\em Ext}^{*}_{C^*(M^{4})}(C^*(M^{2}), C^*(M^2))_{{\Delta^2}^*, {\Delta^2}^*}\ar[d]^{\text{\em Ext}_{1}(1,\Delta^*)}\\
&\text{\em Ext}^*_{C^*(M^{4})}(C^*(M^ 2), C^*(M))_{{\Delta^2}^*,(\Delta^2\circ\Delta)^*}\\
\text{\em Ext}^*_{C^*(M^{2})}(C^*(M), C^*(M))_{\Delta^*, \Delta^*}
&\text{\em Ext}^*_{C^*(M^{3})}(C^*(M), C^*(M))_{((1\times \Delta)\circ \Delta)^*, ((1\times \Delta)\circ \Delta)^*}.
\ar[u]_{\text{\em Ext}_{(1\times \Delta \times 1)^*}({\Delta}^*,1)}^{\cong} \ar[l]^-{\text{\em Ext}_{p_{13}^*}(1, 1)}
}
}
$$
See Remark~\ref{definition generalized V-product} below for the definition of $\widetilde{\vee}$.
The cap with a representative $\sigma$ of the fundamental class $[M]\in H_m(M)$ gives a quasi-isomorphism of right-$C^*(M)$-modules
of upper degre $-m$,
$$\sigma \cap \text{--}  : C^*(M)\buildrel{\simeq}\over\rightarrow C_{m-*}(M), x\mapsto
\sigma\cap x.$$
Let $\Phi:H^{*+m}(LM)\buildrel{\cong}\over\rightarrow
\text{\em Tor}_{C^*(M^{\times 2})}^*(C_*(M), C^*(M))$ be the composite of the isomorphisms
$$
{\footnotesize 
\xymatrix@C15pt@R25pt{ 
H^{p+m}(LM)\ar[r]^-{EM^{-1}}_-\cong
& \text{\em Tor}_{C^*(M^{2})}^{p+m}(C^*(M), C^*(M^I))_{\Delta^*,p^*} \ar[r]^{\text{\em Tor}_{1}(1,\sigma^*)}_\cong
&\text{\em Tor}_{C^*(M^{2})}^{p+m}(C^*(M), C^*(M))_{\Delta^*,\Delta^*}\ar[d]^{\text{\em Tor}_{1}(\sigma \cap \text{--},1)}_\cong\\
&&\text{\em Tor}_{C^*(M^{2})}^p(C_*(M), C^*(M)).
}
}
$$
Then the dual of $\Phi$, $\Phi^\vee: \text{\em Ext}^{-p}_{C^*(M^{2})}(C^*M, C^*M)_{\Delta^*, \Delta^*}\rightarrow H_{p+m}(LM)$
is an isomorphism which respects the multiplication defined here and the loop product.
\end{thm}

\begin{rem}\label{definition generalized V-product}
The isomorphism $\widetilde{\vee}$ in
Theorem~\ref{thm:freeloopExt}
is the composite
$$
{\footnotesize 
\xymatrix@C15pt@R25pt{ 
 \text{Ext}^*_{C^*(M^{2})}(C^*M, C^*M)^{\otimes 2}
\ar[r]^-\vee
& \text{Ext}^*_{C^*(M^{2})^{\otimes 2}}(C^*(M)^{\otimes 2},
C^*(M)^{\otimes 2})\ar[d]^{\text{Ext}_1(1,\gamma)}\\
\text{Ext}^*_{(C_*(M^{2})^{\otimes 2})^\vee}((C_*(M)^{\otimes 2})^\vee,
(C_*(M)^{\otimes 2})^\vee)
\ar[r]_-{\text{Ext}_\gamma(\gamma,1)}^-\cong
\ar[d]_{\text{Ext}_{EZ^\vee}(EZ^\vee,1)}^-\cong
&\text{Ext}^*_{C^*(M^{2})^{\otimes 2}}(C^*(M)^{\otimes 2},
(C_*(M)^{\otimes 2})^\vee)\\
\text{Ext}^*_{C^*(M^{4})}(C^*(M^{2}), (C_*(M)^{\otimes 2})^\vee)
&\text{Ext}^*_{C^*(M^{4})}(C^*(M^{2}), C^*(M^2))
\ar[l]_{\text{Ext}_{1}(1,EZ^\vee)}^-\cong
}
}
$$
where $\vee$ is the $\vee$-product of
Cartan-Eilenberg~\cite[XI. Proposition 1.2.3]{CartanEilenberg}
or~\cite[VIII.Theorem 4.2]{MacLanehomology}, $EZ:C_*(M)^{\otimes
  2}\buildrel{\simeq}\over\rightarrow C_*(M^2)$ denotes the Eilenberg-Zilber
quasi-isomorphism
 and $\gamma:\text{Hom}(C_*(M), \K)^{\otimes 2}\rightarrow
 \text{Hom}(C_*(M)^{\otimes 2}, \K)$ is the canonical map.
\end{rem}
\begin{rem}
We believe that the multiplication on $\text{Ext}^*_{C^*(M^{2})}(C^*(M), C^*(M))_{\Delta^*, \Delta^*}$
defined in Theorem~\ref{thm:freeloopExt} coincides with the Yoneda product.
\end{rem}

Denote by $A(M)$ the functorial commutative differential graded
algebra $A_{PL}(M)$; see ~\cite[Corollary 10.10]{F-H-T}.
Let $\varphi:A(M)^{\otimes 2}\buildrel{\simeq}\over\rightarrow A(M^2)$
be the quasi-isomorphism of algebras given by~\cite[Example 2 p. 142-3]{F-H-T}.
Remark that the composite $\Delta^*\circ\varphi$ coincides with the multiplication of $A(M)$.
Remark also that we have an Eilenberg-Moore isomorphism $EM$ for the
functor $A(M)$; see \cite[Theorem 7.10]{F-H-T}.

Replacing the singular cochains over the rationals $C^*(M;\mathbb{Q})$ by the commutative algebra $A_{PL}(M)$ in Theorem~\ref{thm:torsion_product},
we obtain the following theorem. 

\begin{thm}{\em (}Compare with~\cite{F-T:rationalBV}{\em
    )}\label{rational iso of Felix-Thomas Gorenstein}
Let $N$ be a simply-connected Gorenstein space of dimension $n$ and
$N\rightarrow M$ a continuous map to a simply-connected space $M$.
Let $\Phi$ be the map given by the commutative square
$$
\xymatrix@C25pt@R25pt{ 
H^{p+n}(A(L_N M))\ar@{.>}[d]_\Phi
& \text{\em Tor}^{A(M^{2})}_{-p-n}(A(N), A(M^I))_{\Delta^*,p^*} \ar[d]^{\text{Tor}^{1}(1,\sigma^*)}_\cong\ar[l]_-{EM}^-\cong\\
HH_{-p-n}(A(M),A(N))\ar[r]^-{\text{\em Tor}^{\varphi}(1,1)}_-\cong
&\text{\em Tor}^{A(M^{2})}_{-p-n}(A(N), A(M))_{\Delta^*,\Delta^*}.
}
$$
Then the dual 
$
\xymatrix@1{
HH^{-p-n}(A(M),A(N)^\vee)\ar[r]^-{\Phi^\vee}
& H_{p+n}(L_NM;\mathbb{Q})
}
$ to $\Phi$
is an isomorphism of graded algebras with respect to the loop product $Dlp^\vee$
and the generalized cup product
on Hochschild cohomology induced by 
$(\Delta_{A(N)})^\vee:A(N)^\vee\otimes A(N)^\vee\rightarrow A(N)^\vee$
(See Example~\ref{generalized cup product Gorenstein}).
\end{thm}
\begin{cor}\label{rational iso of Felix-Thomas Poincare}
Let $N$ be a simply-connected Poincar\'e duality space of dimension
$n$.
Let $N\rightarrow M$ be a continuous map to a simply-connected space $M$.
Then $HH^{-p}(A(M),A(N))$ is isomorphic as graded algebras to
$H_{p+n}(L_NM;\mathbb{Q})$ with respect to the loop product $Dlp^\vee$ and the
cup product on Hochschild cohomology induced by the morphism of
algebras
$A(f):A(M)\rightarrow A(N)$ (See Remark~\ref{cup product of an algebra
  morphism} and Definition~\ref{cup product Hochschild}(1))
\end{cor}
\begin{rem}
When $M=N$ is  a Poincar\'e duality space,
such an isomorphism of algebras between Hochschild cohomology and Chas-Sullivan loop space homology was first proved
in \cite{F-T:rationalBV} (See also~\cite{Merkulov} over $\mathbb{R}$). But here our isomorphism is explicit since we do not use a Poincar\'e duality DGA model for $A(M)$
given by~\cite{L-S}.
In fact, as explain in~\cite{F-T:rationalBV}, such an isomorphism is an isomorphism of BV-algebras,
since $\Phi$ is compatible with the circle action and Connes boundary map. Here the BV-algebra on
$ HH^{*}(A(M),A(M))$ is given by~\cite[Theorem 18 or Proof of Corollary 20]{Menichi_BV_Hochschild}.
\end{rem}

In the forthcoming paper \cite{K-L}, we discuss the loop (co)products on the classifying space $BG$ of a Lie group 
$G$ by looking at the integration along the fibre $(Comp)^! : H^*(LBG\times_{BG} LBG) \to H^*(LBG)$ of the homotopy fibration 
$G \to LBG\times_{BG}LBG \to BG$. 
In a sequel \cite{KMnoetherian}, we intend to investigate duality on extension groups of the (co)chain complexes of spaces. 
Such discussion enables one to deduce that Noetherian H-spaces are Gorenstein. In adding, 
the loop homology of a Noetherian H-space is considered.

The rest of this paper is organized as follows. Section 3 is devoted to proving 
Theorems \ref{thm:torsion_product}, \ref{thm:torsion_coproduct}, \ref{thm:EMSS} 
and Corollary \ref{cor:trivial coproduct in EMSS}. 
Theorem \ref{thm:loop(co)product} is proved in Section 4. In Section 5, 
we recall the generalized cup product on the Hochschild cohomology defied by appropriate shriek map. 
Section 6 proves Theorems \ref{thm:loop_homology_ss}, \ref{thm:freeloopExt} and 
\ref{rational iso of Felix-Thomas Gorenstein} and Corollary \ref{rational iso
  of Felix-Thomas Poincare}. 
We prove Proposition \ref{prop:loop_homology} and 
discuss the associativity and commutativity of the loop product  
on Poincar\'e duality space in Section 7. 
In the last three sections, Appendix, 
shriek maps on Gorenstein spaces are considered and their important properties, which we use in the body of the paper,  are described. 




\section{Proofs of Theorems \ref{thm:torsion_product}, \ref{thm:torsion_coproduct} and 
\ref{thm:EMSS}}

In order to prove  Theorem \ref{thm:torsion_product}, 
we consider two commutative diagrams   
$$
\xymatrix@C7pt@R10pt{
  & LM \times LM \ar[rr]^{i}  \ar@{->}'[d]^{p\times p}[dd]  
  & & M^I\times M^I \ar[dd]^{p\times p=p^2} \\
LM\times _MLM \ar[rr]_(0.6){j} \ar[ru]^{q} \ar[dd] & & M^I\times_MM^I \ar[ru]_{\widetilde{q}} \ar[dd]^(0.3){ev_0, ev_1, ev_1=u} & \\
   & M\times M \ar@{->}'[r]^(0.6){\Delta \times \Delta}[rr] & & M^{4} \\
M \ar[rr]_{(1\times \Delta)\circ \Delta=v} \ar[ur]^{\Delta} & & M^{3}   \ar[ru]_{1\times \Delta \times 1=w} &
}
\eqnlabel{add-0}
$$
and 
$$
\xymatrix@C7pt@R10pt{
  & LM \times_M LM \ar[ld]_(0.7){Comp} \ar@{->}'[d][dd]  \ar[rr]^{j} & & M^I\times_M M^I \ar[ld]^{Comp=c} \ar[dd]^{(ev_0, ev_1=ev_0, ev_1)=u} \\
LM \ar[rr]^(0.6){k} \ar[dd]_p & & M^I \ar[dd]^(0.3)p & \\
   & M \ar@{->}'[r]^(0.6){(1\times \Delta)\circ \Delta}[rr] \ar@{=}[ld]& & M^{3} \ar[ld]_{p_{13}}\\
M \ar[rr]_\Delta  & & M\times M &
}
\eqnlabel{add-1}
$$
in which front and back squares are pull-back diagrams. Observe that 
the left and right hand side squares in (3.1) are also pull-back diagrams. 
Here $\Delta$ and $k$ denote the diagonal map and the inclusion, respectively. Moreover 
$p_{13} : M^3 \to M^2$ is the projection defined by $p_{13}(x, y, z) =(x, z)$ and  
$Comp : LM\times_MLM \to LM$ stands for the concatenation of loops. 
The cube (3.2) first appeared in~\cite[p. 320]{F-T:rationalBV}.

\begin{proof}[Proof of  Theorem \ref{thm:torsion_product}] 
Consider the diagram 
$$
\xymatrix@C10pt@R8pt{
\text{Tor}^*_{C^*(M^{2})}(C^*(M), C^*(M^I))_{\Delta^*, p^*} 
\ar[r]^(0.45){\text{Tor}_{p_{13}^*}(1, c^*)}   \ar[d]^(0.6){\cong}_(0.6){EM}                   & 
                       \text{Tor}^*_{C^*(M^{3})}(C^*(M), C^*(M^I\times_M M^I))_{v^*, u^*}  
                        \ar@/_/[ldd]_(0.6){EM_1}^(0.6){\cong}  \\
 H^*(LM)  \ar[d]_{Comp^*} \ar@/_5pc/[dd]_{Dlp}&  
                       \text{Tor}^*_{C^*(M^{4})}(C^*(M), C^*(M^I\times M^I))_{(wv)^*, {p^2}^*}  
                             \ar[u]_{\text{Tor}_{w^*}(1, {\widetilde{q}}^*)}^{\cong} 
                                          \ar[d]^{\text{Tor}_{1}(\Delta^!, 1)} \ar[ld]^{EM_2}_{\cong}\\
H^*(LM\times_M LM)    \ar[d]_{H(q^!)} & 
\text{Tor}^{*+m}_{C^*(M^{4})}(C^*(M^{2}), C^*(M^I\times M^I))_{{\Delta^2}^*, {p^2}^*} . 
\ar[dl]^(0.6){\cong}_(0.6){EM_3} \\                                  
H^{*+m}(LM\times LM)
}
\eqnlabel{add-2}
$$
\noindent where $EM$ and $EM_i$ denote the Eilenberg-Moore maps.

The diagram (3.2) is a morphism of pull-backs from the back face to
the front face. 
Therefore the naturality of the Eilenberg-Moore map yields that 
the upper-left triangle is commutative.

We now consider the front square  
and the right-hand side square in the diagram (3.1). 
The squares are pull-back diagrams and hence we have 
a large pull-back one connecting them.  
Therefore the naturality of the Eilenberg-Moore map shows that
the triangle in the center of the diagram (3.3) is commutative.
Thus it follows  that the map $\text{Tor}_{w^*}(1, {\widetilde{q}}^*)$ is an isomorphism.

Let $\varepsilon:{\mathbb B}\buildrel{\simeq}\over\rightarrow C^*(M)$
be a right $C^*(M^2)$-semifree resolution of $C^*(M)$.
By~\cite[Proof of Theorem 2 or Remark p. 429]{F-T}, 
 the following square is commutative in the derived category of right $C^*(LM\times LM)$-modules..
$$
\xymatrix{
C^*(LM\times_M LM)\ar[r]^{q^!} & C^*(LM\times LM)  \\
{\mathbb B}\otimes_{C^*(M^2)}C^*(LM\times LM) \ar[u]^{EM_4}_{\simeq}
\ar[r]_-{\Delta^! \otimes 1_{C^*(LM\times LM) }} & 
C^*(M^2)\otimes_{C^*(M^2)}C^*(LM\times LM) \ar@{=}[u] ^{EM_5}
}
$$
By taking homology, we obtain the top square in the following diagram commutes.
$$
{\footnotesize
\xymatrix@C18pt@R25pt{
H^*(LM\times_M LM)\ar[r]^{H^*(q^!)}
& H^{*+m}(LM\times LM)  \\
\text{Tor}^*_{C^*(M^2)}(C^ *(M),C^*(LM\times LM)) \ar[u]^{EM_4}_{\cong}
\ar[r]_(0.45){\text{Tor}_{1}(\Delta^!, 1)} & 
\text{Tor}^{*+m}_{C^*(M^2)}(C^ *(M^2),C^*(LM\times LM)) \ar@{=}[u]^{EM_5} \\
\text{Tor}^*_{C^*(M^{4})}(C^*(M), C^*(M^I\times M^I))
\ar[r]_(0.45){\text{Tor}_{1}(\Delta^!, 1)} \ar[u]_{\text{Tor}^*_{(\Delta\times\Delta)^*}(1,i^*)}
\ar@/^8pc/[uu]^(0.7){EM_2}
& \text{Tor}^{*+m}_{C^*(M^{4})}(C^*(M^2), C^*(M^I\times M^I))
\ar[u]^{\text{Tor}^*_{(\Delta\times\Delta)^*}(1,i^*)}\ar@/_8pc/[uu]_(0.7){EM_3}
}
}
\eqnlabel{add-3}
$$
The bottom square commutes obviously.
We now consider the left-hand square  
and the back square in the diagram (3.1). 
The squares are pull-back diagrams and hence we have 
a large pull-back one connecting them.  
Therefore the naturality of the Eilenberg-Moore map shows that the left-hand side in (3.4) is commutative.
The same argument or the definition of the Eilenberg-Moore map shows that the right-hand side in (3.4) is commutative.

So finally, the lower square in (3.3) is commutative.

The usual proof~\cite[p. 26]{G-M} that the Eilenberg-Moore isomorphism
$EM$ is an isomorphism of algebras with respect to the cup product
gives
the following commutative square
$$
\xymatrix@C15pt@R10pt{
H^*(LM)^{\otimes 2}\ar[d]^\cong_\times
& \text{Tor}^{*}_{C^*(M^{2})}(C^*(M), C^*(M^I))_{{\Delta}^*, {p}^*}^{\otimes 2} 
\ar[l]_-{\cong}^-{EM^{\otimes 2}}\ar[d]_\cong^{\widetilde{\top}}\\
H^{*}(LM\times LM)
&\text{Tor}^{*}_{C^*(M^{4})}(C^*(M^{2}), C^*(M^I\times M^I))_{{\Delta^2}^*, {p^2}^*} 
\ar[l]^-{\cong}_-{EM_3}.
}
$$
This square is the top square in \cite[p. 255]{Mccleary}.

Consider the commutative diagram of spaces where the three composites of the vertical morphisms are the diagonal maps
$$
\xymatrix@C25pt@R10pt{
M\ar[d]_\sigma^\simeq\ar@/_3pc/[dd]_{\Delta}
\ar@{=}[r]
&M\ar[r]^\Delta\ar[d]_{\sigma'}^\simeq
& M\times M\ar[d]_{\sigma\times\sigma}^\simeq\ar@/^3pc/[dd]^{\Delta\times\Delta}\\
M^I\ar[d]_p
&M^I\times_M M^I\ar[r]_{\widetilde{q}}\ar[d]^u\ar[l]^-c
& M^I\times M^I\ar[d]_{p\times p}\\
M\times M 
&M^3\ar[r]_w\ar[l]^{p_{13}}
& M^4.
}
$$
Using the homotopy equivalence $\sigma$, $\sigma'$ and $\sigma^2$,
we have the result. 
\end{proof}

We decompose the maps, which induce the loop coproduct, with pull-back diagrams. 
Let $l : LM \to M\times M$ 
be a map defined by $l(\gamma)= (\gamma(0), \gamma(\frac{1}{2}))$. 
We define a map $\varphi :LM \to LM$ by $\varphi(\gamma)(t)=\gamma(2t)$ for 
$0\leq t \leq\frac{1}{2}$ and $\varphi(\gamma)(t)=\gamma(1)$ for 
$\frac{1}{2} \leq t \leq 1$. Then $\varphi$
is homotopic to the identity map
and fits into the commutative diagram  
$$
\xymatrix@C10pt@R5pt{
  & LM  \ar[ld]_{\varphi}^{\simeq} \ar@{->}'[d][dd]^(0.4){ev_0}  \ar[rr]^{j} & & M^I \ar[ld]^(0.4){\beta} \ar[dd]^{ev_0, ev_1=p} \\
LM \ar[rr] \ar[dd]_{l} & & M^I \times M^I \ar[dd]^(0.3){p\times p}& \\
   & M \ar@{->}'[r]^(0.7){\Delta}[rr] \ar@{->}[ld]_{\Delta} & & M\times M \ar[ld]^{\alpha} \\
M \times M \ar[rr]_{\gamma'}  & & M^{4}.  &
}
\eqnlabel{add-3}
$$
Here the maps $\alpha : M^{2} \to M^{4}$, $\beta : M^I \to M^I \times M^I$ and 
$\gamma' : M^{2} \to M^{4}$ are defined by $\alpha(x, y) = (x, y, y, y)$, $\beta(r) = (r, c_{r(1)})$ with the constant loop 
$c_{r(1)}$ at $r(1)$ and 
$\gamma'(x, y) =(x, y, y, x)$, respectively. 
We consider moreover the two pull-back squares
$$
\xymatrix@C15pt@R15pt{
LM\times _M LM \ar[r]^-{Comp}\ar[d]
& LM\ar[r]\ar[d]_l
& M^I\times M^I\ar[d]^{p\times p}\\
M\ar[r]_\Delta
& M\times M\ar[r]_{\gamma'}
& M^4
}
\eqnlabel{add-4}
$$
and the commutative cube
$$
\xymatrix@C10pt@R8pt{
  & LM \times_M LM \ar[ld]_(0.55){q} \ar@{->}'[d][dd]  \ar[rr]^{i\circ
  q} & & M^I\times M^I 
  \ar@{=}[ld] \ar[dd]^{p^2} \\
LM \times LM \ar[rr]^(0.6){i} \ar[dd] & & M^I \times M^I \ar[dd]^(0.3){p^2} & \\
   & M \ar@{->}'[r]^(0.6){(\Delta\times \Delta)\circ \Delta}[rr] \ar@{->}[ld]_{\Delta} 
                                    & & M^{4} \ar@{=}[ld] \\
M \times M\ar[rr]_{\Delta \times \Delta}  & & M^{4} &
}
\eqnlabel{add-5}
$$
in which 
front and back squares are also pull-back diagrams.  

\begin{proof}[Proof of Theorem \ref{thm:torsion_coproduct}] We see that 
the diagrams (3.5), (3.6) and (3.7) give rise to a commutative diagram
$$
\hspace{-0.5cm}
{
\xymatrix@C12pt@R20pt{
 H^*(LM\times LM)  \ar[d]_{q^*} \ar@/_4pc/[dd]_{Dlcop}&  
\text{Tor}^*_{C^*(M^{4})}(C^*(M^2), C^*((M^I)^2))_{{\Delta^2}^*,{p^2}^*}   
                                          \ar[d]^{\text{Tor}_{1}(\Delta^*, 1)} \ar[l]_-{EM}^-{\cong}\\
H^*(LM\times_M LM)    \ar[d]_{Comp^!} & 
 \text{Tor}^*_{C^*(M^{4})}(C^*(M), C^*(M^I\times M^I))_{(wv)^*, {p^2}^*}  
                                          \ar[d]^{\text{Tor}_{1}(\Delta^!, 1)} \ar[l]^-{EM}_-{\cong} \\
H^{*+m}(LM) \ar[d]_{\varphi^*=id} & 
\text{Tor}^{*+m}_{C^*(M^{4})}(C^*(M^2), C^*(M^I\times M^I))_{{\gamma'}^*, {p^2}^*} . 
\ar[l]^-{\cong}_-{EM} \ar[d]_{\text{Tor}_{\alpha^*}(\Delta^*, \beta^*)}^{\cong}
 \\
H^{*+m}(LM) &
\text{Tor}^{*+m}_{C^*(M^{2})}(C^*(M), C^*(M^I))_{{\Delta}*,{p}^*}. \ar[l]^-{EM}_-{\cong}
}
}
$$
In fact, the diagrams (3.5) and (3.7) give morphisms of pull-backs from the back face to
the front face. 
Therefore the naturality of the Eilenberg-Moore map yields that 
the top and the bottom squares are commutative.

Using the diagram (3.6), the same argument as in the proof 
of Theorem \ref{thm:torsion_product} enables us to conclude that  the
middle square is commutative. 

Since the following diagram of spaces 
$$
\xymatrix@C15pt@R18pt{
M\ar[r]^\Delta\ar[d]_\sigma^\simeq\ar@/_3pc/[dd]_{\Delta}
& M\times M\ar[d]_{\sigma\times\sigma}^\simeq\ar@/^3pc/[dd]^{\Delta\times\Delta}\\
M^I\ar[r]^\beta\ar[d]_p
& M^I\times M^I\ar[d]_{p\times p}\\
M\times M \ar[r]^\alpha
& M^4 
}
$$
is commutative, 
the theorem follows.
\end{proof}

By considering the free loop fibration
$\Omega M\buildrel{\widetilde{\eta}}\over\rightarrow
LM\buildrel{ev_0}\over\rightarrow M$, we define for Gorenstein space
(see Example~\ref{intersection fibration}) an intersection
morphism $H(\widetilde{\eta}_!):H_{*+m}(LM)\rightarrow H_*(\Omega M)$
generalizing the one defined by Chas and Sullivan~\cite{C-S}.
Using the following commutative cube
$$
\xymatrix@C10pt@R8pt{
  & LM\ar[rr]  \ar@{->}'[d]^{ev_0}[dd]  
  & & M^I \ar[dd]^{p} \\
\Omega M \ar[rr] \ar[ru]^{\widetilde{\eta}} \ar[dd]
&& PM \ar[ru]_{\widetilde{\eta\times 1}} \ar[dd]^(0.3){ev_1} & \\
   & M \ar@{->}'[r]^(0.6){\Delta}[rr] & &
   M\times M \\
{*} \ar[rr] \ar[ur]^{\eta} & &
{*}\times M  \ar[ru]_{\eta \times 1} &
}
$$
where all the faces are pull-backs, we obtain similarly the following theorem. 
\begin{thm}
Let $M$ be a simply-connected Gorenstein with generator
$\omega_M$ in $\text{\em Ext}^{m}_{C^*(M)}(\K,C^*(M))$. Then the dual of the intersection
morphism $H(\widetilde{\eta}^!)$ is given by the commutative diagram
$$
\xymatrix@C12pt@R20pt{
H^*(\Omega M)\ar[dd]_{H(\tilde{\eta}^!)}
& \text{\em Tor}^*_{C^*(M)}(\K,C^*(PM))_{{\eta}*,{ev_1}^*}
\ar[l]_-{EM}^-{\cong}\ar[r]^{\text{\em Tor}^*_{1}(1,\eta^*)}_\cong
& \text{\em Tor}^*_{C^*(M)}(\K,\K)_{{\eta}*,{\eta}^*}\\
& \text{\em Tor}^*_{C^*(M^2)}(\K,C^*(M^I))_{{\eta}*,{p}^*}
\ar[lu]^-{EM}_-{\cong}\ar[d]^{\text{\em Tor}^*_{1}(\omega_M,1)}
\ar[u]_{\text{\em Tor}^*_{(\eta\times 1)^*}(1,(\widetilde{\eta\times 1})^*)}^\cong
\ar[r]^{\text{\em Tor}^*_{1}(1,\sigma^*)}_\cong
&\text{\em Tor}^*_{C^*(M^2)}(\K,C^*M)_{{\eta}*,{\Delta}^*}
\ar[d]^{\text{\em Tor}^*_{1}(\omega_M,1)}
\ar[u]_{\text{\em Tor}^*_{(\eta\times 1)^*}(1,\eta^*)}^\cong\\
H^{*+m}(LM)
& \text{\em Tor}^*_{C^*(M^2)}(C^*(M),C^*(M^I))_{{\Delta}*,{p}^*}\ar[l]^-{EM}_-{\cong}
\ar[r]^{\text{\em Tor}^*_{1}(1,\sigma^*)}_\cong
&\text{\em Tor}^*_{C^*(M^2)}(C^*M,C^*M)_{{\Delta}*,{\Delta}^*}.
}
$$
\end{thm}

\medskip
Let $\widehat{\mathcal F}$ be the pull-back diagram in the front of (3.1). Let  
$\widetilde{\mathcal F}$ denote the pull-back diagram obtained by combining the front and 
the right hand-side squares in (3.1).  
Then a map inducing the isomorphism 
$\text{Tor}_{w^*}(1, {\widetilde{q}}^*)$ gives rise to 
a morphism $\{f_r\} : \{\widetilde{E}_r, \widetilde{d}_r \} \to \{\widehat{E}_r, \widehat{d}_r \}$ 
of spectral sequences, where 
$\{\widehat{E}_r, \widehat{d}_r \}$ and  $\{\widetilde{E}_r, \widetilde{d}_r \}$ are the 
Eilenberg-Moore spectral sequences associated with the fibre squares 
$\widehat{\mathcal F}$ and $\widetilde{\mathcal F}$, respectively. 
In order to prove Theorem \ref{thm:EMSS}, we need the following lemma.

\begin{lem}
\label{lem:key}
The map  $f_2$ is an isomorphism. 
\end{lem}

\begin{proof}
We identify $f_2$ with the map  
$$
\text{Tor}_{H^*(w)}(1, H^*(\Delta)): 
 \text{Tor}_{H^*(M^4)}(H^*M, H^*(M\times M))\rightarrow
 \text{Tor}_{H^*(M^3)}(H^*M, H^*M).
$$
\noindent
up to isomorphism between the $E_2$-term and the torsion product.
Thus, in order to obtain the result, it suffices to
apply part (1) of Lemma~\ref{lem:generalizeddecompositioncup} for the algebra $H^ *(M)$ and the module $H^*(M)$. 
\end{proof}

We are now ready to give the EMSS (co)multiplicative structures. 

\medskip
\noindent
{\it Proof of Theorem \ref{thm:EMSS}.}
Gugenheim and May~\cite[p. 26]{G-M} have shown that the map
$\widetilde{\top}$
induces a morphism of spectral sequences from $E_r\otimes E_r$ 
to the Eilenberg-Moore spectral sequence converging to $H^*(LM\times
LM)$.
In fact, $\widetilde{\top}$ induces an isomorphism of spectral sequences.
All the other maps between torsion products in  
Theorems \ref{thm:torsion_product} and \ref{thm:torsion_coproduct} preserve the filtrations. 
Thus in view of Lemma \ref{lem:key}, we have Theorem \ref{thm:EMSS}.

In fact, the shriek map 
$\Delta^!$ is in $\text{Ext}_{C^*(M^2)}^m(C^*(M), C^*(M^2))$. Then we have 
$d\Delta^! = (-1)^m\Delta^!d$. 
Let $\{\widehat{E}_r^{*,*} , \widehat{d}_r\}$ and $\{\widetilde{E}_r^{*,*} , \widetilde{d}_r\}$ 
be the EMSS's converging to $\text{Tor}_{C^*(M^4)}^*(C^*(M), C^*((M^I))^2)$ and 
$\text{Tor}_{C^*(M^4)}^*(C^*(M^2), C^*((M^I)^2))$, respectively. Let 
$\{f_r\} : \{\widehat{E}_r^{*,*} , \widehat{d}_r\} \to
\{\widetilde{E}_r^{*,*} , \widetilde{d}_r\}$ be the morphism of
spectral sequences which gives rise to $\text{Tor}_1(\Delta^!, 1)$. 
Recall the map $\Delta^! \otimes 1 : {\mathbb B}\otimes_{C^*(M^4)}{\mathbb B}' \to 
C^*(M^2)\otimes_{C^*(M^4)}{\mathbb B}'$ in the proof of 
Theorem \ref{thm:torsion_product}.  
It follows that, for any $b\otimes b' \in {\mathbb B}\otimes_{C^*(M^4)}{\mathbb B}'$, 
\begin{eqnarray*}
(\Delta^!\otimes 1)d(b\otimes b' )&=& \Delta^!\otimes 1(db\otimes b' + (-1)^{\deg b}b\otimes db') \\
&=&\Delta^!db\otimes b' + (-1)^{\deg b}\Delta^! b\otimes db' \\
&=& (-1)^md\Delta^!b\otimes b' + (-1)^{\deg b}\Delta^! b\otimes db'.
\end{eqnarray*}
On the other hand, we see that 
\begin{eqnarray*}
d(\Delta^!\otimes 1)(b\otimes b' )&=& d(\Delta^!(b\otimes b')) \\
&=&d\Delta^!b\otimes b' + (-1)^{\deg b +m}\Delta^! b\otimes db' 
\end{eqnarray*}
and hence $(\Delta^!\otimes 1)d=(-1)^m d (\Delta^!\otimes 1)$. 
This implies that $f_r\widehat{d}_r = (-1)^m\widetilde{d}_rf_r$. 
The fact yields  the compatibility of the multiplication 
with the differential of the spectral sequence. 

The same argument does work well to show the compatibility of the comultiplication with the differential of the EMSS. 
\hfill\qed

\medskip
\noindent
{\it Proof of Corollary \ref{cor:trivial coproduct in EMSS}.}
Since $H^*(\Delta^!)$ is $H^*(M^2)$-linear, it follows that 
$H^*(\Delta^!)\circ H^*(\Delta)(x)=H^*(\Delta^!)(1)\cup x$.
If $d<0$ then $H^*(\Delta^!)(1)=0$.

If $d=0$ then  $H^*(\Delta^!)(1)=\lambda 1$ where $\lambda\in \K$ and so the composite
$H^*(\Delta^!)\circ H^*(\Delta)$ is the multiplication by the scalar  $\lambda$.
Let $m$ be an non-trivial element of positive degre in $H^*(M)$.
Then we see that $0=H^*(\Delta^!)\circ H^*(\Delta)(m\otimes 1-1\otimes m)=\lambda(m\otimes 1-1\otimes m).$
Therefore $\lambda=0$.

So in both cases, we have proved that $H^*(\Delta^!)\circ H^*(\Delta)=0$.
Since $H^*(\Delta)$ is surjective, $H^*(\Delta^!):H^*(M)\rightarrow H^{*+d}(M^2)$ is trivial.
In particular, the induced maps  
$\text{Tor}^{*}_{H^*(M^{4})}(H^*(\Delta^!), H^*(M^I\times_M M^I))$
and
$\text{Tor}^{*}_{H^*(M^{4})}(H^*(\Delta^!), H^*((M^I)^2)$
are trivial.
Then it follows from Theorems~\ref{thm:torsion_product}
and \ref{thm:torsion_coproduct} that both the comultiplication and the multiplication on the $E_2$-term of the EMSS, 
which correspond to the duals to loop
product and loop coproduct on $H^*(LM)$ are null.
Therefore, $E^{*,*}_\infty\cong \text{Gr}H^*(LM)$ is equipped with a trivial coproduct and a trivial product. 
Then the conclusion follows.
\hfill\qed
\begin{rem}
It follows from Corollary \ref{cor:trivial coproduct in EMSS}
that under the hypothesis of Corollary \ref{cor:trivial coproduct in EMSS}
the two composites
$$
\xymatrix@C30pt@R5pt{
H^*(M)\otimes H^*(M) \ar[r]^(0.48){p^*\otimes p^*} & H^*(LM)\otimes H^*(LM)
\ar[r]^(0.6){Dlcop}  &H^*(LM)  
}
$$
and 
$$
\xymatrix@C25pt@R5pt{
H_*(LM)\otimes H_*(LM)
\ar[rr]^(0.6){\text{Loop product}} & &H_*(LM)
\ar[r]^{p_*} & H_*(M)
}
$$
are trivial. 
This can also be proved directly since
we have the commuting diagram
$$
\xymatrix@C30pt@R15pt{
H^*(LM\times_M LM)\ar[r]^{H(Comp^!)}
&H^*(LM\times LM)\\
H^*(M)\ar[r]_{H(\Delta^!)}\ar[u]
&H^*(M\times M)\ar[u]_{(p \times p) ^*}
}
$$
and since $p_*:H_*(LM)\rightarrow H_*(M)$ is a morphism of graded algebras
with respect to the loop product and to the intersection product
$H(\Delta_!)$. As we saw in the proof of Corollary \ref{cor:trivial coproduct in EMSS}, under the hypothesis of Corollary
\ref{cor:trivial coproduct in EMSS},
$H^*(\Delta^!)$ and its dual $H_*(\Delta_!)$ are trivial.
\end{rem}

\section{Proof of Theorem \ref{thm:loop(co)product}} 
The following Lemma is interesting on his own since it gives a very
simple proof of a result of Klein (see Remark~\ref{Klein} below).  
\begin{lem}\label{isoextfreeloops}
Let $M$ be an oriented simply-connected Poincar\'e duality space of dimension $m$.
Let $M\twoheadrightarrow B$ be a fibration.
Denote by $M\times_B M$ the pull-back over $B$.
Then for all $p\in {\Z}$,
$\text{\em Ext}^{-p}_{C^*(B)}(C^*(M),C^*(M))$ is isomorphic to $H_{p+m}(M\times_B M)$ as a vector space.
\end{lem}

\begin{rem}
A particular case of Lemma~\ref{isoextfreeloops} is the isomorphim of graded vector spaces
$\text{Ext}^{-p}_{C^*(M\times M)}(C^*(M),C^*(M))\cong H_{p+m}(LM)$ underlying the isomorphism of algebras
given in Theorem~\ref{thm:freeloopExt}.
Note yet that in the proof of Lemma~\ref{isoextfreeloops}, we consider right $C^*(B)$-modules and that in the proof of Theorem~\ref{thm:freeloopExt}, we need left $C^*(M^2)$-modules; see Section 10. 
\end{rem}

\begin{rem}\label{Klein}
Let $F$ be the homotopy fibre of $M\rightarrow B$.
In~\cite[Theorem B]{KleinPoincare}, Klein shows in term of spectra that
$
\text{Ext}^{-p}_{C_*(\Omega B)}(C_*(F),C_*(F))\cong H_{p+m}(M\times_B M)
$
and so, using the Yoneda product~\cite[Theorem A]{KleinPoincare}, $H_{*+m}(M\times_B M)$ is a graded algebra.

The isomorphism above and that in Lemma  \ref{isoextfreeloops} make us aware of {\it duality} 
on the extension functors of (co)chain complexes of spaces. As mentioned in 
the Introduction, this is one of topics in \cite{KMnoetherian}. 
\end{rem}

\begin{proof}[Proof of Lemma \ref{isoextfreeloops}]
The Eilenberg-Moore map gives an isomorphism
$$
H_{p+m}(M\times_B M)\cong \text{Ext}^{-p-m}_{C^*(B)}(C^*(M),C_*(M)).$$
The cap with a representative $\sigma$ of the fundamental class $[M]\in H_m(M)$ gives a quasi-isomorphism of right-$C^*(M)$-modules
of upper degre $-m$,
$$\sigma \cap \text{--} : C^*(M)\buildrel{\simeq}\over\rightarrow C_{m-*}(M), x\mapsto
\sigma\cap x.$$
Therefore, we have an isomorphism
$$
\text{Ext}^{*}_{C^*(B)}(C^*(M),\sigma \cap \text{--}):\text{Ext}^{-p}_{C^*(B)}(C^*(M),C^*(M))\rightarrow\text{Ext}^{-p-m}_{C^*(B)}(C^*(M),C_*(M))
$$
This completes the proof. 
\end{proof}
\begin{proof}[Proof of Theorem \ref{thm:loop(co)product}]
Theorems   \ref{thm:torsion_product}
and \ref{thm:torsion_coproduct} allow us to describe part of the composite
$H^*(LM\times_M LM)\buildrel{H(q^!)}\over\rightarrow H^{*+m}(LM\times LM)\buildrel{Dlcop}\over\rightarrow H^{*+m}(LM\times LM)$
 in terms of the following composite of appropriate maps between torsion functors 
$$
{
\xymatrix@C30pt@R10pt{  
                       \text{Tor}^*_{C^*(M^{4})}(C^*(M), C^*(M^2))_{(wv)^*, {\Delta^2}^*}                                
                                          \ar[d]^{\text{Tor}_{1}(\Delta^!, 1)} \\
\text{Tor}^{*+m}_{C^*(M^{4})}(C^*(M^{2}), C^*(M^2))_{{\Delta^2}^*, {\Delta^2}^*} . 
\ar[d]^(0.45){\text{Tor}_{1}(\Delta^*, 1)}              \\ 
                 \text{Tor}^{*+m}_{C^*(M^{4})}(C^*(M), C^*(M^2))_{(wv)^*, {\Delta^2}^*} 
                        \ar[d]^{\text{Tor}_{1}(\Delta^!, 1)} \\
                       \text{Tor}^{*+2m}_{C^*(M^{4})}(C^*(M^2), C^*(M^2))_{{\gamma'}^*, {\Delta^2}^*}  
\ar[d]_(0.45){\text{Tor}_{1}(\Delta^*, 1)}                                  
\ar@/^7mm/[rd]^(0.7){\text{Tor}_{\alpha^*}(\Delta^*, \Delta^*)}_(0.7)\cong \\
\text{Tor}^{*+2m}_{C^*(M^{4})}(C^*(M), C^*(M^2))_{\Delta^*{\alpha}^*, {\Delta^2}^*}
\ar[r]_{\text{Tor}_{\alpha^*}(1, \Delta^*)}
& \text{Tor}^{*+2m}_{C^*(M^{2})}(C^*(M), C^*(M))_{\Delta^*, \Delta^*} . 
}
}
$$
By virtue of Lemma~\ref{isoextfreeloops}, we see that 
$\text{Ext}^{2m}_{C^*(M^4)}(C^*(M),C^*(M))_{(wv)^*,\Delta^*{\alpha}^*}$
is isomorphic to $H_{-m}(M^{S^1\vee S^1\vee S^1})=\{0\}$.
Then the composite
$C^*(M)\buildrel{\Delta^!}\over\rightarrow C^*(M\times M)\buildrel{\Delta^*}\over\rightarrow C^*(M)\buildrel{\Delta^!}\over\rightarrow C^*(M\times M)
\buildrel{\Delta^*}\over\rightarrow C^*(M)
$
is null in $\D(\text{Mod-}C^*(M^4))$. 
Therefore the composite $Dlcop\circ H(q^!)$ is trivial and hence 
$Dlcop\circ Dlp:=Dlcop\circ H(q^!)\circ comp^*$ is also trivial.
\end{proof}
\begin{rem}

Instead of using Lemma~\ref{isoextfreeloops}, one can show that
$$\text{Ext}^{2m}_{C^*(M^4)}(C^*(M),C^*(M))_{(wv)^*,\Delta^*{\alpha}^*}=\{0\}$$ as
follow:
Consider the cohomological Eilenberg-Moore spectral sequence with 
$${\mathbb E}^{p,*}_2\cong\text{Ext}^p_{H^*(M^4)}(H^*(M),H^*(M))$$ converging to 
$\text{Ext}^{*}_{C^*(M^4)}(C^*(M),C^*(M))_{(wv)^*,\Delta^*{\alpha}^*}$. Then we see that 
$${\mathbb E}^{p,*}_1=\text{Hom}(H^*(M)\otimes H^+(M^4)^{\otimes p}, H^*(M)).$$
Therefore, since $M^4$ is simply-connected and $H^{>m}(M)=\{0\}$,
${\mathbb E}^{p,q}_r=\{0\}$ if $q>m-2p$ (Compare with Remark~\ref{zeroelementsofEMSS}).
Therefore $\text{Ext}^{p+q}_{C^*(M^4)}(C^*(M),C^*(M))_{(wv)^*,\Delta^*{\alpha}^*}=\{0\}$ if $p+q>m$.
\end{rem}

\begin{rem}\label{rem:obstruction}
Let $M$ be a Gorenstein space of dimension $m$. The proof of Theorem
2.12 shows that if the composite 
$$\Delta^*\circ\Delta^!\circ
\Delta^*\circ\Delta^!\in\text{Ext}^{2m}_{C^*(M^4)}(C^*(M),C^*(M))_{(wv)^*,\Delta^*{\alpha}^*}$$
is the zero element. Then  $Dlcop\circ Dlp$ trivial.
\end{rem}

\begin{rem}
In the proof of Theorem \ref{thm:loop(co)product}, it is important to
work in the derived category of $C^*(M^4)$-modules:
Suppose that $M$ is the classifying space of a connected Lie group of
dimension $-m$. Then since $m$ is negative,
the composite $\Delta^!\circ \Delta^*$ is null. In fact $\Delta^!\circ \Delta^* \in
\text{Ext}^{m}_{C^*(M^2)}(C^*(M^2),C^*(M^2))_{1^*,1^*}\cong H^m(M^2)=\{0\}$.  
But in general, $Dlcop$ is not trivial; see \cite[Theorem D]{F-T} and \cite{K-L}. 
Therefore the composite $\Delta^*\circ\Delta^!\circ \Delta^*\in
\text{Ext}^{m}_{C^*(M^4)}(C^*(M^2),C^*(M))_{{\Delta^2}^*,\Delta^*{\alpha}^*}$
is also not trivial.
\end{rem}

\section{The generalized cup product on the Hochschild cohomology}

After recalling (defining) the (generalized) cup product on the Hochschild cohomology, we give an extension functor description 
of the product. The result plays an important role in proving our main theorem, Theorem \ref{thm:loop_homology_ss}. 

\begin{defn}\label{cup product Hochschild}
Let $A$ be a (differential graded) algebra.
Let $M$ be a $A$-bimodule.
Recall that we have a canonical
map~\cite[p. 283]{Menichi_BV_Hochschild}
$$
\otimes_A:HH^*(A,M)\otimes HH^*(A,M)\rightarrow HH^*(A,M\otimes_A M).
$$

{(1)} Let $\bar{\mu}_M:M\otimes_A M\rightarrow M$ be a morphism of
 $A$-bimodules of degree $d$.
Then the {\it cup product} $\cup$ on $HH^*(A,M)$ is the composite
$$
HH^p(A,M)\otimes HH^q(A,M)\buildrel{\otimes_A}\over\rightarrow
HH^{p+q}(A,M\otimes_A
M)\buildrel{HH^{p+q}(A,\bar{\mu}_M)}\over\longrightarrow HH^{p+q+d}(A,M).
$$

{(2)} Let $\varepsilon:Q\buildrel{\simeq}\over\rightarrow M\otimes_A M$
be a $A\otimes A^{op}$-projective (semi-free) resolution of $M\otimes_A M$.
Let $\bar{\mu}_M\in\text{Ext}_{A\otimes A^{op}}^d(M\otimes_A
M,M)=H^d(\text{Hom}_{A\otimes A^{op}}(Q,M))$.
Then the {\it generalized cup product} $\cup$ on $HH^*(A,M)$ is the composite
$$
HH^*(A,M)^{\otimes 2}\buildrel{\otimes_A}\over\rightarrow
HH^{*}(A,M\otimes_A M)
\buildrel{HH^{*}(A,\varepsilon)^{-1}}\over\longrightarrow HH^{*}(A,Q)
\buildrel{HH^{*}(A,\bar{\mu}_M)}\over\longrightarrow HH^{*}(A,M).
$$
\end{defn}
\begin{rem}\label{cup product of an algebra morphism}
Let $M$ be an associative (differential graded) algebra with unit
$1_M$.
Let $h:A\rightarrow M$ be a morphism of (differential graded)
algebras.
Then $$a\cdot m\star b:=h(a)mh(b)$$ defines an $A$-bimodule structure
on $M$
such that the multiplication of $M$, $\mu_M:M\otimes M\rightarrow M$
induces a morphism of $A$-bimodules $\bar{\mu}_M:M\otimes_A
M\rightarrow M$.

Conversely, let $M$ be a $A$-bimodule equipped with an element $1_M\in
M$ and a morphism of $A$-bimodules $\bar{\mu}_M:M\otimes_A
M\rightarrow M$ such that
$\bar{\mu}_M\circ(\bar{\mu}_M\otimes_A
1)=\bar{\mu}_M\circ(1\otimes_A\bar{\mu}_M)$ and such that the two maps
$m\mapsto \bar{\mu}_M(m\otimes_A 1)$ and $m\mapsto
\bar{\mu}_M(1\otimes_A m)$
coincide with the identity map on $M$.
Then the map $h:A\rightarrow M$ defined by $h(a):=a\cdot 1_M$ is
a morphism of algebras.
\end{rem}
The following lemma gives an interesting decomposition of the cup
product of the Hochschild cohomology of a commutative (possible
differential graded) algebra.

\begin{lem}\label{lem:generalizeddecompositioncup}
Let $A$ be a commutative (differential graded) algebra. Let $M$ be a
$A$-module. Let $B$ be an $A^{\otimes 2}$-module.
 Let $\mu:A^{\otimes 2}\rightarrow A$ 
denote the multiplication of $A$. Let $\eta:\K\rightarrow A$ be the unit of $A$.
Let $q:B\otimes B\twoheadrightarrow B\otimes_A B$ be the quotient map.
Then

{\em (1)}  $\text{\em Tor}^{1\otimes\mu\otimes 1}_*(1,\mu):\text{\em Tor}^{A^{\otimes 4}}_*(M,A\otimes A)\buildrel{\cong}\over\rightarrow \text{\em Tor}^{A^{\otimes 3}}_*(M,A)$
is an isomorphism,

{\em (2)} $\text{\em Hom}_{1\otimes\mu\otimes 1}(q,1):\text{\em Hom}_{A^{\otimes
    3}}(B\otimes_A
B,M)\buildrel{\cong}\over\rightarrow\text{\em Hom}_{A^{\otimes
    4}}(B\otimes B,M)$ is an  isomorphism and
 
{\em (3)} $\text{\em Ext}_{1\otimes\mu\otimes 1}^*(q,1): 
\text{\em Ext}_{A^{\otimes 3}}^*(B\otimes_A B,M)\buildrel{\cong}\over\rightarrow\text{\em Ext}_{A^{\otimes 4}}^*(B\otimes B,M)$ is also an  isomorphism.

{\em (4)} Let $\mu_M\in\text{Hom}_{A^{\otimes 4}}(M^{\otimes 2},M)$.
Then $\mu_M$ induced a quotient map
$\bar{\mu}_M:M\otimes_A M\rightarrow M$
and
the cup product $\cup$ of the Hochschild cohomology of $A$ with coefficients in $M$,
$HH^*(A,M)=\text{\em Ext}^*_{A^{\otimes 2}}(A,M)_{\mu,\mu}$ is given by the following
commutative diagram

$
\xymatrix@C12pt@R10pt{
\text{\em Ext}^*_{A^{\otimes 2}}(A,M)_{\mu,\mu}\otimes
\text{\em Ext}^*_{A^{\otimes 2}}(A,M)_{\mu,\mu}\ar[r]^-{\otimes}\ar[dd]_\cup
&\text{\em Ext}^*_{A^{\otimes 4}}(A^{\otimes 2},M^{\otimes 2})_{\mu^{\otimes 2},\mu^{\otimes 2}}
\ar[d]^{\text{\em Ext}^*_{1}(1,\mu_M)}\\
&\text{\em Ext}^*_{A^{\otimes 4}}(A^{\otimes 2},M)_{\mu^{\otimes 2},\mu\circ \mu^{\otimes 2}}
\ar[d]^{\text{\em Ext}^*_{1\otimes\mu\otimes 1}(\mu,1)^{-1}}_\cong\\
\text{\em Ext}^*_{A^{\otimes 2}}(A,M)_{\mu,\mu}&\text{\em Ext}^*_{A^{\otimes 3}}(A,M)_{\mu\circ(\mu\otimes 1),\mu\circ(\mu\otimes 1)}
\ar[l]^-{\text{\em Ext}^*_{1\otimes\eta\otimes 1}(1,1)}
}
$

{\em (5)} Let $\varepsilon:R\buildrel{\simeq}\over\rightarrow M\otimes M$
be a $A^{\otimes 4}$-projective (semi-free) resolution of $M\otimes M$.
Let $\mu_M\in\text{\em Ext}_{A^{\otimes 4}}(M^{\otimes
  2},M)=H(\text{\em Hom}_{A^{\otimes 4}}(R,M))$.
Let $\bar{\mu}_M$ be $\text{\em Ext}_{1\otimes\mu\otimes 1}^*(q,1)^{-1}(\mu_M)$. 
Then 
the generalized cup product $\cup$ of the Hochschild cohomology of $A$ with coefficients in $M$,
$HH^*(A,M)=\text{\em Ext}^*_{A^{\otimes 2}}(A,M)_{\mu,\mu}$ is given by the following
commutative diagram
$$
\xymatrix@C12pt@R10pt{
\text{\em Ext}^*_{A^{\otimes 2}}(A,M)_{\mu,\mu}\otimes
\text{\em Ext}^*_{A^{\otimes 2}}(A,M)_{\mu,\mu}\ar[r]^-{\otimes}\ar[ddd]_\cup
&\text{\em Ext}^*_{A^{\otimes 4}}(A^{\otimes 2},M^{\otimes 2})_{\mu^{\otimes 2},\mu^{\otimes 2}}
\ar[d]^{(\text{\em Ext}^*_{1}(1,\varepsilon))^{-1}}_\cong\\
&\text{\em Ext}^*_{A^{\otimes 4}}(A^{\otimes 2},R)
\ar[d]^{\text{\em Ext}^*_{1}(1,\mu_M)}\\
&\text{\em Ext}^*_{A^{\otimes 4}}(A^{\otimes 2},M)_{\mu^{\otimes 2},\mu\circ \mu^{\otimes 2}}
\ar[d]^{\text{\em Ext}^*_{1\otimes\mu\otimes 1}(\mu,1)^{-1}}_\cong\\
\text{\em Ext}^*_{A^{\otimes 2}}(A,M)_{\mu,\mu}&\text{\em Ext}^*_{A^{\otimes 3}}(A,M)_{\mu\circ(\mu\otimes 1),\mu\circ(\mu\otimes 1)}
\ar[l]^-{\text{\em Ext}^*_{1\otimes\eta\otimes 1}(1,1)}
}
$$
\end{lem}

As mentioned at the beginning of this section, Lemma \ref{lem:generalizeddecompositioncup} (4) contributes toward proving 
Theorem \ref{thm:loop_homology_ss}. Moreover in view of part (5) of the lemma, 
we prove Theorem \ref{rational  iso of Felix-Thomas Gorenstein}.

\begin{proof}[Proof of Lemma \ref{lem:generalizeddecompositioncup}]
(1) Consider the bar resolution $\xi: B(A, A, A) \stackrel{\simeq}{\to} A$ of $A$. 
Since the complex $B(A, A, A)$ is a semifree $A$-module, it follows from 
\cite[Theorem 6.1]{F-H-T} that 
$\xi\otimes_A \xi : B(A, A, A)\otimes_A B(A, A, A) \to A\otimes_A A=A$ is a quasi-isomorphism and hence it is a projective resolution of $A$ as a $A^{\otimes 3}$-module. 
We moreover have a commutative diagram 
$$
\xymatrix@C15pt@R15pt{
B(A, A, A)\otimes B(A, A, A) \ar[r]^(0.7){\xi\otimes \xi} \ar@{->>}[d]_q & A \otimes A \ar[d]^\mu \\
B(A, A, A)\otimes_A B(A, A, A) \ar[r]_(0.75){\xi\otimes_A \xi} & A
}
$$
in which $q$ is the natural projection and the first row is a projective resolution of $A\otimes A$ as a 
$A^{\otimes 4}$-module.
It is immediate that
$q$ is a morphism of $A^{\otimes 4}$-modules with respect to the
morphism of algebras $1\otimes \mu \otimes 1:A^{\otimes 4}\rightarrow A^{\otimes 3}$.
Then $\text{Tor}_{1\otimes \mu \otimes 1}(1, \mu)$ is induced by the map
$$
1\otimes q : M\otimes_{A^{\otimes 4}}B(A, A, A)\otimes B(A, A, A) \to 
M\otimes_{A^{\otimes 3}}B(A, A, A)\otimes_A B(A, A, A). 
$$
Since $A$ is a commutative, it follows that both the source and target of $1\otimes u$ are isomorphic to $W:= M\otimes B(\K, A, \K) \otimes B(\K, A, \K)$ as a vector space. 
As a linear map, $1\otimes q$ coincides with the identity map on $W$ up to isomorphism.

(2) By the universal property of the quotient map $q:B\otimes B\twoheadrightarrow
B\otimes_A B$, $\text{Hom}_{1\otimes\mu\otimes 1}(q,1)$ is an isomorphism.

(3) Let $\varepsilon:P\buildrel{\simeq}\over\rightarrow B$ be an $A^{\otimes
  2}$-projective
(semifree) resolution of $B$.
We have a commutative square of $A^{\otimes 4}$-modules
$$
\xymatrix@C20pt@R15pt{
P\otimes P\ar[r]^{\varepsilon\otimes\varepsilon}_\simeq\ar[d]_{q'}
&B\otimes B \ar[d]^q\\
P\otimes_A P\ar[r]^{\varepsilon\otimes_A\varepsilon}_\simeq
&B\otimes_A B
}
$$
Therefore $\text{Ext}_{1\otimes\mu\otimes 1}^*(q,1)$ is induced
by $\text{Hom}_{1\otimes\mu\otimes 1}(q',1)$ which is an isomorphism
by (2).

(4) Let $A$ be any algebra and $M$ be any $A$-bimodule.
Let $\xi:\mathbb{B}\buildrel{\simeq}\over\rightarrow A$ an $A\otimes A^{op}$-projective
(semi-free) resolution (for example the double bar resolution).
Let $c:\mathbb{B}\rightarrow \mathbb{B}\otimes_A\mathbb{B}$ be a morphism of $A$-bimodules
such that the diagram of $A$-bimodules
$$
\xymatrix@C20pt@R15pt{
\mathbb{B}\ar[r]^\xi\ar[d]^c
&A \ar[d]_\cong\\
\mathbb{B}\otimes_A\mathbb{B}\ar[r]_{\xi\otimes_A\xi}
&A\otimes_A A
}
$$
is homotopy commutative.
The cup product of $f$ and $g\in\text{Hom}_{A\otimes A^{op}}(\mathbb{B},M)$ is the composite
$\bar{\mu}_M\circ (f\otimes_A g)\circ c\in \text{Hom}_{A\otimes A^{op}}(\mathbb{B},M)$~\cite[p. 134]{S-G}.

Suppose now that $A$ is commutative and that the $A$-bimodule
structure on $M$ comes from the multiplication $\mu$ of $A$ and an $A$-module structure on $M$.
The following diagram of complexes gives two different decompositions
of the cup product on $\text{Hom}_{A\otimes A^{op}}(\mathbb{B},M)$.
$$
{
\xymatrix@C20pt@R10pt{
\text{Hom}_{A^{\otimes 2}}(\mathbb{B},M)\otimes \text{Hom}_{A^{\otimes 2}}(\mathbb{B},M)
\ar[r]^-{\otimes}\ar[d]^{\otimes_A}
& \text{Hom}_{A^{\otimes 4}}(\mathbb{B}\otimes\mathbb{B},M\otimes M)
\ar[d]^{\text{Hom}_{1}(1,\mu_M)}\\
\text{Hom}_{A^{\otimes 3}}(\mathbb{B}\otimes_A\mathbb{B},M\otimes_A M)
\ar[dr]^{\text{Hom}_{1}(1,\bar{\mu}_M)}
&\text{Hom}_{A^{\otimes 4}}(\mathbb{B}\otimes\mathbb{B},M)\\
\text{Hom}_{A\otimes A^{op}}(\mathbb{B}\otimes_A\mathbb{B},M)
\ar[d]^{\text{Hom}_{1}(c,1)}
&\text{Hom}_{A^{\otimes 3}}(\mathbb{B}\otimes_A\mathbb{B},M)
\ar[u]_{\text{Hom}_{1\otimes\mu\otimes 1}(q,1)}^\cong
\ar[l]^{\text{Hom}_{1\otimes\eta\otimes 1}(1,1)}\\
\text{Hom}_{A\otimes A^{op}}(\mathbb{B},M)
}
}
$$
(5) Let $\varepsilon:P\buildrel{\simeq}\over\rightarrow M$
be a surjective $A\otimes A^{op}$-projective (semifree) resolution
of $M$.
Tnen $\mu_M$ can be considered as an element of
$\text{Hom}_{A^{\otimes 4}}(P\otimes P, M)$.
By lifting, there exists $\mu_P\in\text{Hom}_{A^{\otimes 4}}(P\otimes
P, P)$
such that $\varepsilon\circ\mu_P=\mu_M$.
By 2), there exists $\bar{\mu}_P$ such that $\bar{\mu}_P\circ
q=\mu_P$.
We can take $\bar{\mu}_M=\varepsilon\circ\bar{\mu}_P$.

It is now easy to check that the isomorphism
$HH^*(A,\varepsilon):HH^*(A,P)\buildrel{\cong}\over\rightarrow
HH^*(A,M)$ 
transports the cup product on $HH^*(A,P)$ defined using $\bar{\mu}_P$
to the generalized cup product on $HH^*(A,M)$ defined using
$\bar{\mu}_M$.
We now check that the isomorphism $HH^*(A,\varepsilon):HH^*(A,P)\buildrel{\cong}\over\rightarrow
HH^*(A,M)$ transports the composite
$$
\xymatrix@C15pt@R10pt{
\text{Ext}^*_{A^{\otimes 2}}(A,P)_{\mu,\mu}\otimes
\text{Ext}^*_{A^{\otimes 2}}(A,P)_{\mu,\mu}\ar[r]^-{\otimes}
&\text{Ext}^*_{A^{\otimes 4}}(A^{\otimes 2},P^{\otimes 2})_{\mu^{\otimes 2},\mu^{\otimes 2}}
\ar[d]^{\text{Ext}^*_{1}(1,\mu_P)}\\
&\text{Ext}^*_{A^{\otimes 4}}(A^{\otimes 2},P)_{\mu^{\otimes 2},\mu\circ \mu^{\otimes 2}}
\ar[d]^{\text{Ext}^*_{1\otimes\mu\otimes 1}(\mu,1)^{-1}}_\cong\\
\text{Ext}^*_{A^{\otimes 2}}(A,P)_{\mu,\mu}&\text{Ext}^*_{A^{\otimes 3}}(A,P)_{\mu\circ(\mu\otimes 1),\mu\circ(\mu\otimes 1)}
\ar[l]^-{\text{Ext}^*_{1\otimes\eta\otimes 1}(1,1)}
}
$$
into the composite
$$
\xymatrix@C15pt@R10pt{
\text{Ext}^*_{A^{\otimes 2}}(A,M)_{\mu,\mu}\otimes
\text{Ext}^*_{A^{\otimes 2}}(A,M)_{\mu,\mu}\ar[r]^-{\otimes}
&\text{Ext}^*_{A^{\otimes 4}}(A^{\otimes 2},M^{\otimes 2})_{\mu^{\otimes 2},\mu^{\otimes 2}}
\ar[d]^{(\text{Ext}^*_{1}(1,\varepsilon\otimes\varepsilon))^{-1}}_\cong\\
&\text{Ext}^*_{A^{\otimes 4}}(A^{\otimes 2},P\otimes P)
\ar[d]^{\text{Ext}^*_{1}(1,\mu_M)}\\
&\text{Ext}^*_{A^{\otimes 4}}(A^{\otimes 2},M)_{\mu^{\otimes 2},\mu\circ \mu^{\otimes 2}}
\ar[d]^{\text{Ext}^*_{1\otimes\mu\otimes 1}(\mu,1)^{-1}}_\cong\\
\text{Ext}^*_{A^{\otimes 2}}(A,M)_{\mu,\mu}&\text{Ext}^*_{A^{\otimes 3}}(A,M)_{\mu\circ(\mu\otimes 1),\mu\circ(\mu\otimes 1)}
\ar[l]^-{\text{Ext}^*_{1\otimes\eta\otimes 1}(1,1)}
}
$$
By applying (4) to $\mu_P$, we have proved (5).
\end{proof}
\begin{thm}(Compare with~\cite[Theorem 12]{F-T})\label{ext Gorenstein commutatif}
Let $B$ be a simply-connected commutative Gorenstein cochain algebra
of dimension $m$ such that $\forall i\in\mathbb{N}$, $H^i(B)$ is
finite dimensional.
Then
$
\text{\em Ext}^{*+m}_{B\otimes B}(B,B\otimes B)\cong H^*(B).
$
\end{thm}
\begin{proof}
The proof of~\cite[Theorem 12]{F-T} for the strongly homotopy
commutative algebra $C^*(X)$ obviously works in the case of a
commutative
algebra $B$.
\end{proof}
\begin{rem}
In~\cite[Theorem 2.1 i) iv)]{A-I}
Avramov and Iyengar have shown a related result in the non graded case:
Let $S$ be a a commutative algebra over a field $\K$, which is the quotient of a polynomial
algebra $\K[x_1,\dots,x_d]$ or more generally which is the quotient of a localization
of $\K[x_1,\dots,x_d]$. Then $S$ is Gorenstein if and only if the graded $S$-module
$
\text{Ext}^{*}_{S\otimes S}(S,S\otimes S)
$
is projective of rank $1$.
\end{rem}
\begin{ex}\label{generalized cup product Gorenstein}(The generalized cup product of a Gorenstein algebra)
Let $A\rightarrow B$ be a morphism of commutative differential graded
algebras where $B$ satisfies the hypotheses of
Theorem~\ref{ext Gorenstein commutatif}.
Let $\Delta_B:B\rightarrow B\otimes B$ be a generator of
$\text{Ext}^{m}_{B\otimes B}(B,B\otimes B)\cong \K$.
By taking duals, we obtain the following element of
$
\text{Ext}^{m}_{A\otimes A}(B^\vee\otimes B^\vee,B^\vee)
$:
$$
(\Delta_B)^{\vee}:B^\vee\otimes B^\vee\rightarrow (B\otimes B)^\vee
\rightarrow B^\vee.
$$
By 3) of Lemma~\ref{lem:generalizeddecompositioncup},
$
(\Delta_B)^{\vee}$ induces an element
$\bar{\mu}_{B^{\vee}}\in\text{Ext}^{m}_{A\otimes A}(B^\vee\otimes_A
B^\vee,B^\vee)$.
Therefore, by ii) of definition~\ref{cup product Hochschild},
we have a generalized cup product
$$
HH^p(A,B^\vee)\otimes HH^q(A,B^\vee)
\buildrel{\cup}\over\rightarrow HH^{p+q+m}(A,B^\vee).
$$
In the case $A=B$, of course, we believe that
$HH^*(A,A^\vee)$ equipped with this generalized cup product
and Connes coboundary is a non-unital BV-algebra.
\end{ex}

\section{Proofs of Theorems  \ref{thm:loop_homology_ss},  \ref{thm:freeloopExt} and \ref{rational
    iso of Felix-Thomas Gorenstein} and Corollary~\ref{rational
    iso of Felix-Thomas Poincare}}  

We first prove Theorem \ref{rational iso of Felix-Thomas Gorenstein}. The same argument is applicable when proving 
Theorem \ref{thm:loop_homology_ss}. 

\begin{proof}[Proof of Theorem \ref{rational iso of Felix-Thomas Gorenstein}]
Step 1:
The polynomial differential functor $A(X)$ extends to a functor
$A(X,Y)$ for pairs of spaces $Y\subset X$.
The two natural short exact sequences~\cite[p. 124]{F-H-T}
$\rightarrow A(X,Y)\rightarrow A(X)\rightarrow A(Y)\rightarrow 0$
and $0\rightarrow C^*(X,Y)\rightarrow C^*(X)\rightarrow
C^*(Y)\rightarrow 0$
are naturally weakly equivalent~\cite[p. 127-8]{F-H-T}.
Therefore all the results of Felix and Thomas given in~\cite{F-T}
with the singular cochains algebra $C^*(X)$ are valid
with $A(X)$ (For example, the description of the shriek map of an
embedding $N\hookrightarrow M$ at the level of singular cochains
given page 419 of~\cite{F-T}).
In particular, our Theorem~\ref{thm:torsion_product}
is valid when we replace $C^*(X)$ by $A(X)$.
(Note also that a proof similar to the proof of Theorems~\ref{thm:newThm3} or~\ref{produitshriek} shows that the dual of the loop product on $A(L_NM)$ is isomorphic
to the dual of the loop product defined on $C^*(L_NM)$.)
This means the following:
Let $\Delta^!_A$ be a generator of
$\text{Ext}_{A(N^2)}^n(A(N),A(N^2))\cong \Q$ given by \cite[Theorem 12]{F-T}. Then the composite $\text{Tor}^1(1,\sigma^*)\circ EM^{-1}$
is an isomorphism of algebras between the dual of the loop product
$Dlp$
on $H^*(A(L_NM))$ and the coproduct defined by the composite on the left
column of the following diagram.

Step 2: We have chosen $\Delta^!_A$ and $\Delta_{A(N)}$ such that the composite
$\varphi\circ\Delta_{A(N)}$ is equal to $\Delta^!_A$ in the derived category of
$A(N)^{\otimes 2}$-modules. Therefore the following diagram commutes.
$$
\xymatrix@C50pt@R15pt{
\text{Tor}^{A(M^{2})}_{*}(A(N), A(M))\ar[d]_{\text{Tor}^{A(p_{13})}(1,1)}
&\text{Tor}^{A(M)^{\otimes 2}}_{*}(A(N), A(M))\ar[l]_-{\text{Tor}^{\varphi}(1,1)}^-\cong
\ar[d]_{\text{Tor}^{1\otimes\eta\otimes1}(1,1)}\\
\text{Tor}^{A(M^{3})}_{*}(A(N), A(M))
&\text{Tor}^{A(M)^{\otimes 3}}_{*}(A(N), A(M))\ar[l]_-{\text{Tor}^{\varphi\circ(1\otimes\varphi)}(1,1)}^-\cong\\
\text{Tor}^{A(M^{4})}_{*}(A(N), A(M^2))\ar[u]^{\text{Tor}^{A(1\times\Delta\times 1)}(1,A(\Delta))}_\cong
\ar[d]_{\text{Tor}^{1}(\Delta^!_A,1)}
&\text{Tor}^{A(M)^{\otimes 4}}_{*}(A(N), A(M)^{\otimes 2})\ar[l]_-{\text{Tor}^{\varphi}(1,\varphi)}^-\cong
\ar[u]^{\text{Tor}^{1\otimes\mu\otimes1}(1,\mu)}_\cong
\ar[d]_{\text{Tor}^{1}(\Delta_{A(N)},1)}\\
\text{Tor}^{A(M^{4})}_{*}(A(N^2), A(M^2))\ar[d]_\cong
&\text{Tor}^{A(M)^{\otimes 4}}_{*}(A(N)^{\otimes 2}, A(M)^{\otimes 2})\ar[l]_-{\text{Tor}^{\varphi}(\varphi,\varphi)}^-\cong
\ar[d]_\cong\\
\left(\text{Tor}^{A(M^{2})}_{*}(A(N), A(M))\right)^{\otimes 2}
&\left(\text{Tor}^{A(M)^{\otimes 2}}_{*}(A(N), A(M))\right)^{\otimes 2}\ar[l]^-{\left(\text{Tor}^{\varphi}(1,1)\right)^{\otimes 2}}_-\cong
}
$$
Step 3:
Dualizing and using the natural isomorphism
$$
\text{Ext}^*_B(Q,P^\vee)\buildrel{\cong}\over\rightarrow \text{Tor}^*_B(P,Q)^\vee
$$
for any differential graded algebra $B$, right $B$-module $P$ and left $B$-module $Q$,
we see that the dual of $\Phi$ is an isomorphism of algebras with respect to the loop product and to the long composite
given by the diagram of (5) of Lemma~\ref{lem:generalizeddecompositioncup} when $A:=A(M)$, $M:=A(N)^\vee$ and $\mu_M:=(\Delta_{A(N)})^\vee$.

Step 4: We apply part (5) of Lemma~\ref{lem:generalizeddecompositioncup} to see that this long composite coincides with 
the generalized cup product of the Gorenstein algebra $B:=A(N)$.  
\end{proof}

\noindent
{\it Proof of Theorem  \ref{thm:loop_homology_ss}.}  
Let $\{E_r^{*,*}, d_r\}$ denote the spectral sequence described in Theorem \ref{thm:EMSS}.
Then we define a spectral sequence $\{{\mathbb E}_r^{*,*}, d_r^\vee\}$ by 
$$
{\mathbb E}_r^{p,q+d}:= ({E_r^{*,*}}^\vee)^{p,q} = (E_r^{-p,-q})^\vee. 
$$ 
The decreasing filtration $\{F^pH^*(L_N M) \}_{p\leq 0}$ of $H^*(L_N M)$ induces the decreasing filtration $\{F^pH_*(L_N M) \}_{p\geq 0}$ of $H_*(L_N M)$ defined by 
$$F^pH_*(L_N M)= (H^*(L_N M)/F^{-p}H^*(L_N M))^\vee.$$ 
By definition, the Chas-Sullivan loop homology (the shift homology) ${\mathbb H}_*(L_N M)$ is given by 
${\mathbb H}_{-(p+q)}(L_N M)=(H^*(L_N M)^\vee)^{p+q-d}$. By Proposition~\ref{prop:loop_homology},
the product $m$ on ${\mathbb H}_*(L_N M)$ is defined by 
$$
m(a\otimes b)= (-1)^{d(|a|-d)}(Dlp)^\vee(a\otimes b)
$$
for $a\otimes b \in (H^*(L_N M)^\vee)^{*}\otimes (H^*(L_N M)^\vee)^{*}$.  
Then we see that 
\begin{eqnarray*}
{\mathbb E}_{\infty}^{p, q}  &\cong& F^p(H^*(L_N M)^\vee)^{p+q-d}/ F^{p+1}(H^*(L_N M)^\vee)^{p+q-d} \\ 
&=&F^p{\mathbb H}_{-(p+q)}(L_N M)/F^{p+1}{\mathbb H}_{-(p+q)}(L_N M).
\end{eqnarray*}

The composite  $\widetilde{(Dlp)}$ in 
Theorem \ref{thm:torsion_product} which gives rise to $Dlp$ on
$H^*(L_N M)$ preserves the filtration of the EMSS $\{E^{*,*}_r, d_r\}$; see Remark \ref{rem:relative_cases}. 
As mentioned in the proof of Theorem \ref{thm:EMSS}, 
the map  $\widetilde{(Dlp)}$ induces the morphism  $(Dlp)_r : E_r^{*, *} \to E_r^{*, *}\otimes E_r^{*, *}$  
of spectral sequences of bidegree $(0, d)$. 
Define $m_r  : {\mathbb E}_r^{*, *} \otimes {\mathbb E}_r^{*, *}\to {\mathbb E}_r^{*, *}$ by 
$$
m_r(a\otimes b)=(-1)^{d(|a|+d)}((Dlp)_r)^\vee(a\otimes b), 
$$  
where $|a| = p+ q$ if $a \in {({E_r^{*,*}}^\vee)}^{p, q}$. Then a straightforward computation enables us to deduce that
$m_r(d_r^\vee a\otimes b+(-1)^{\vert a\vert+d}a \otimes d_r^\vee b) = d_r^\vee\circ m_r(a\otimes b)$
for any $r$. Note that $\vert a\vert+d=p+q+d$ is the total degree of $a$ in ${\mathbb E}_r^{*, *}$.
It turns out that 
$\{{\mathbb E}_r, d_r^\vee\}$ is a spectral sequence of algebras converging to ${\mathbb H}_{-*}(L_N M)$ as an algebra. 

%

It remains now to identify the $\mathbb{E}_2$-term with 
the Hochschild cohomology. We proceed as in the proof of Theorem~\ref{rational iso
  of Felix-Thomas Gorenstein} replacing the polynomial differential functor $A$
  by singular cohomology $H^*$.
The product $Dlp_2$ is given by the composite 
$$
\xymatrix@C50pt@R15pt{
\text{Tor}^{H^*(M^{2})}_{*}(H^*(N), H^*(M))\ar[r]_{\text{Tor}^{H^*(p_{13})}(1,1)}&
\text{Tor}^{H^*(M^{3})}_{*}(H^*(N), H^*(M))\\
&\text{Tor}^{H^*(M^{4})}_{*}(H^*(N), H^*(M^2))\ar[u]_{\text{Tor}^{H^*(1\times\Delta\times 1)}(1,H^*(\Delta))}^\cong
\ar[d]_{\text{Tor}^{1}(\Delta^!,1)}\\
\left(\text{Tor}^{H^*(M^{2})}_{*}(H^*(N), H^*(M))\right)^{\otimes 2}
&\text{Tor}^{H^*(M^{4})}_{*}(H^*(N^2), H^*(M^2))\ar[l]_\cong
}
$$
Dualizing and using the natural isomorphism
$$
\text{Ext}^*_B(Q,P^\vee)\buildrel{\cong}\over\rightarrow \text{Tor}^*_B(P,Q)^\vee
$$
for any graded algebra $B$, right $B$-module $P$ and left $B$-module $Q$,
we see that $Dlp_2^\vee$, the dual of $Dlp_2$ is the long composite
given by the diagram of Lemma~\ref{lem:generalizeddecompositioncup} (4) when $A:=H^*(M)$, $M:=H_*(N)$ and $\mu_M:=H(\Delta^!)^\vee$.
By Lemma~\ref{lem:generalizeddecompositioncup} (4), we obtain
an isomorphism of algebras 
$
u : HH^*(H^*(M),H_*(N)) \stackrel{\cong}{\to} (E^{*,*}_2)^\vee
$
with respect to $Dlp_2^\vee$ and the cup product induced by
$$
\bar{\mu}_M=\overline{H(\Delta^{!})^\vee}:H_*(N)\otimes_{H^{*}(M)}H_*(N)\rightarrow H_*(N).
$$
Using Example~\ref{lem:important_degree} (ii) and Example~\ref{examples naturality cup product} (i), we finally obtain
an isomorphism of algebras
$\mathbb{E}^{*,*}_2\cong HH^*(H^*(M),\mathbb{H}_*(N))$
with respect to $m_2$ and the cup product induced by
$$
\bar{\mu}_\mathbb{M}:\mathbb{H}_*(N)\otimes_{H^{*}(M)}\mathbb{H}_*(N)\rightarrow \mathbb{H}_*(N).
$$
Note that $\bar{\mu}_\mathbb{M}$ coincides with $\overline{H(\Delta^{!})^\vee}$ only up to the multiplication by $(-1)^d$. 

Suppose further that $N$ is a Poincar\'e duality space of dimension
$d$.
Consider the two squares
$$
\xymatrix@C30pt@R15pt{
H_{d-p}(N)\otimes H_{d-q}(N)\ar[r]^-\times
& H_{2d-p-q}(N\times N)\ar[r]^-{H(\Delta^!)^\vee}
& H_{d-p-q}(N)\\
H^p(N)\otimes H^q(N)\ar[r]^-\times\ar[u]^{-\cap [N] \otimes -\cap [N]}
& H^{p+q}(N\times N)\ar[r]^-{H(\Delta)}\ar[u]_{- \cap [N] \times [N]}
& H^{p+q}(N)\ar[u]_{-\cap [N]}
}
$$
The right square is the diagram 2) of Proposition~\ref{shriek and Poincare duality}. 
Therefore the right square commutes by
Corollary~\ref{shriek diagonal}.
By~\cite[VI.5.4 Theorem]{Bredongeometry}, we see that 
$$(\alpha\times\beta)\cap ([N] \times[N])=
(-1)^{\vert\alpha\vert\vert [N]\vert}
  (\alpha\cap[N])\times (\beta\cap[N])=\times\circ(-\cap
  [N])\otimes(-\cap [N])(\alpha\otimes \beta).
$$
This means that the left square commutes.
Therefore we have proved that the isomorphism of lower degree $d$,
$\theta_{H^*(N)}:=- \cap [N]:H^*(N)\rightarrow H_{d-*}(N)$ is a morphism of algebras
with respect to the cup product and the composite of
$H(\Delta^!)^\vee$ and the homological cross product.
By naturality of the cup product on Hochschild cohomology defined
(Remark~\ref{cup product of an algebra morphism}) by a morphism of algebras,
this implies that the morphism 
$$
HH^*(1,- \cap [N]):HH^*(H^*(M),H^*(N))\buildrel{\cong}\over
\rightarrow HH^*(H^*(M),H_*(N))
$$
is an isomorphism of algebras of lower degree $d$.
We see that the composite 
$$
\zeta = u\circ HH^*(1,- \cap [N]):HH^*(H^*(M),H^*(N))\buildrel{\cong}\over
\rightarrow  {\mathbb E}_2^{*,*}
$$
is an isomorphism of algebras; see Example~\ref{lem:important_degree} i) and ii). 
This completes the proof.  
\hfill\qed

\begin{rem}\label{zeroelementsofEMSS}
For the EMSS $\{E_r^{*,*}, d_r\}$ described in Theorem \ref{thm:EMSS}, 
we see that $E_r^{p,q}=0$ if $q< - 2p$ 
since $M$ is simply-connected. This implies that ${\mathbb E}_r^{p, q}=0$ if $q > -2p+d$. 
\end{rem}

\begin{proof}[Proof of Corollary~\ref{rational iso of Felix-Thomas
    Poincare}]
Denote by $\int_N:C^*(N)\buildrel{\simeq}\over\rightarrow A(N)$
a quasi-isomorphism of complexes which coincides in homology with the
natural equivalence of algebras between the singular cochains and the
polynomial differential forms~\cite[Corollary 10.10]{F-H-T}.
Let $\psi_N:C_*(N)\hookrightarrow C^*(N)^\vee$ be
the canonical inclusion of the complexe $C_*(N)$ into its bidual
defined in~\cite[7.1]{F-T-VP} or~\cite[Property 57 i)]{Menichi} by $\psi_N(c)(\varphi)=(-1)^{\vert c\vert\vert\varphi\vert}\varphi(c)$
for $c\in C_*(N)$ and $\varphi\in C^*(N)$.
Consider the diagram of complexes
$$
\xymatrix@C30pt@R15pt{
C_*(N)\otimes C_*(N)\ar[rr]^{EZ}\ar[d]_{\psi_N\otimes\psi_N}
&& C_*(N\times N)\ar[d]^{\psi_{N\times N}}\\
C^*(N)^\vee\otimes C^*(N)^\vee\ar[r]\ar[d]_{\int_N^\vee\otimes\int_N^\vee}
& (C^*(N)\otimes C^*(N))^\vee\ar[d]_{(\int_N\otimes\int_N)^\vee}
& C^*(N\times N)^\vee\ar[l]_-{AW^{\vee\vee}}
\ar[d]^{\int_{N\times N}^\vee}\\
A(N)^\vee\otimes A(N)^\vee\ar[r]
&(A(N)\otimes A(N))^\vee
& A(N\times N)^\vee\ar[l]^-{\varphi^\vee}
}
$$
where $EZ$ and $AW$ are the Eilenberg-Zilber and Alexander-Whitney
maps. The bottom left square commutes by naturality of the horizontal
maps.
The top rectangle commutes in homology since
by~\cite[VI.5.4 Theorem]{Bredongeometry},
$<\alpha\times\beta,a\times b>=(-1)^{\vert\beta\vert\vert a\vert}
<\alpha,a><\beta,b>$
for all $\alpha$, $\beta\in H^*(N)$ and $a$, $b\in H_*(N)$.
The bottom right square commutes in homology since using $H(\int_N)$,
$H(\varphi)$ can be identified with the cohomological cross product 
~\cite[Example 2 p. 142-3]{F-H-T} and~\cite[Chap. 5 Sec. 6 14
Corollary]{Spanier}.

Let $\theta_N:A(N)\buildrel{\simeq}\over\rightarrow A(N)^\vee$ be a
quasi-isomorphim of upper degree $-n$ right $A(N)$-linear such that
the image of the fundamental class $[N]$ by the composite
$C_*(N)\buildrel{\psi_N}\over\rightarrow
C^*(N)^{\vee}\buildrel{\int_N^\vee}\over\rightarrow A(N)^\vee$
is the class of $\theta_N(1)$.
Let $\theta_{N\times N}:A(N\times N)\buildrel{\simeq}\over\rightarrow
A(N\times N)^\vee$ be a
quasi-isomorphim of upper degree $-2n$ right $A(N^2)$-linear such that
the image of the fundamental class $[N]\times[N]$ by the composite
$C_*(N\times N)\buildrel{\psi_{N\times N}}\over\rightarrow
C^*(N\times N)^{\vee}\buildrel{\int_{N\times N}^\vee}\over\rightarrow
A(N\times N)^\vee$
is the class of $\theta_{N\times N}(1)$.
Using the previous commutative diagram, the classes
$\varphi^\vee\circ \theta_{N\times N}(1)$ and
$\theta_N(1)\otimes\theta_N(1)$
are equal.

Consider the diagram in the derived category of $A(N)^{\otimes 2}$-modules
$$
\xymatrix{
A(N)\otimes A(N)\ar[rr]^\varphi\ar[d]_{\theta_N\otimes\theta_N}
&& A(N\times N)\ar[r]^-{A(\Delta)}\ar[d]_{\theta_{N\times N}}
& A(N)\ar[d]^{\theta_N}\\
A(N)^\vee\otimes A(N)^\vee\ar[r]
& (A(N)\otimes A(N))^\vee
& A(N\times N)^\vee\ar[r]_-{\Delta^{!\vee}_A}\ar[l]^-{\varphi^\vee}
& A(N)^\vee
}
$$
By Corollary~\ref{shriek diagonal}, the right square commutes in the derived category
of
$A(N\times N)$-modules.
The left rectangle commutes up to homotopy of $A(N)\otimes
A(N)$-modules
since the classes $\varphi^\vee\circ \theta_{N\times N}(1)$ and
$\theta_N(1)\otimes\theta_N(1)$
are equal.

Finally, since $\theta_N:A(N)\buildrel{\simeq}\over\rightarrow
A(N)^\vee$
is a morphism of algebras of upper degree $-n$ in the derived category
of $A(N)^{\otimes 2}$-modules, by example~\ref{examples naturality cup product}ii),
$HH^*(1,\theta_N):HH^*(A(M),A(N))\buildrel{\cong}\over\rightarrow
HH^*(A(M),A(N)^\vee)$ is an isomorphism of algebras of upper degree $-n$.
\end{proof}
\begin{proof}[Proof of Theorem \ref{thm:freeloopExt}]
Let $\varepsilon:\mathbb{B}\buildrel{\simeq}\over\rightarrow C^*(M)$
be a right $C^*(M^2)$-semifree resolution of $C^*(M)$.
By 1) of Proposition~\ref{shriek and Poincare duality} and Corollary
~\ref{shriek diagonal},
$\Delta^!$ fits into
the following homotopy commutative diagram of right
$C^*(M^2)$-modules.
$$
\xymatrix@C30pt@R15pt{
\mathbb{B}\ar[rr]^{\Delta^!}\ar[d]_{\varepsilon}^\simeq
&& C^{*+m}(M^2)\ar[d]^{\sigma^2 \cap \text{ --}}_\simeq\\
C^*(M) \ar[r]^{\sigma \cap \text{ --}}_\simeq
& C_{m-*}(M)\ar[r]^{\Delta_*}
& C_{m-*}(M^2)
}
$$
Here $(\sigma^2 \cap \text{ --})(x)=EZ(\sigma\otimes\sigma)\cap x$.
By applying the functor
$\text{Tor}^*_{C^*(M^4)}(-,C^*(M^2))$,
we obtain the commutative square
$$
\xymatrix@C30pt@R15pt{
\text{Tor}^*_{C^*(M^4)}(C^*(M),C^*(M^2))
\ar[rr]^{\text{Tor}_1(\sigma \cap \text{ --},1)}\ar[d]_{\text{Tor}_1(\Delta^!,1)}
&&\text{Tor}^*_{C^*(M^4)}(C_*(M),C^*(M^2))
\ar[d]^{\text{Tor}_1(\Delta_*,1)}\\
\text{Tor}^*_{C^*(M^4)}(C^*(M^2),C^*(M^2))\ar[rr]_{\text{Tor}_1(\sigma^2 \cap \text{ --},1)}
&&\text{Tor}^*_{C^*(M^4)}(C_*(M^2),C^*(M^2)).
}
$$
Therefore using Theorem~\ref{thm:torsion_product}, $\Phi$ is an
isomorphism
of coalgebras with respect to the dual of the loop product and to the
following composite
$$
\xymatrix@C30pt@R15pt{
\text{Tor}^*_{C^*(M^{2})}(C_*M, C^*M)
\ar[r]^-{\text{Tor}_{p_{13}^*}(1, 1)}              
&  \text{Tor}^*_{C^*(M^{3})}(C_*M, C^*M)\\
\text{Tor}^{*}_{C^*(M^{4})}(C_*(M^{2}), C^*(M^2))
&               \text{Tor}^*_{C^*(M^{4})}(C_*(M), C^*(M^2)) 
\ar[u]_{\text{Tor}_{(1\times \Delta \times 1)^*}(1, {\Delta}^*)}^{\cong} 
\ar[l]^{\text{Tor}_{1}(\Delta_*, 1)}\\
\text{Tor}^*_{C^*(M^{4})}(C_*(M)^{\otimes 2},C^*(M^2))
\ar[r]_-{\text{Tor}_{EZ^\vee}(1,EZ^\vee)}^-\cong
\ar[u]_-{\text{Tor}_{1}(EZ,1)}^-\cong
& \text{Tor}^*_{(C_*(M^{2})^{\otimes 2})^\vee}(C_*(M)^{\otimes 2},
(C_*(M)^{\otimes 2})^\vee)\\
\text{Tor}^*_{C^*(M^{2})}(C_*(M), C^*(M))^{\otimes 2}\ar[r]^-{\cong}_-\top
& \text{Tor}^*_{C^*(M^{2})^{\otimes 2}}(C_*(M)^{\otimes 2},
C^*(M)^{\otimes 2}).
\ar[u]^{\text{Tor}_{\gamma}(1,\gamma)}_-\cong
}
$$
Dualizing and using the natural isomorphism
$$
\text{Ext}^*_B(Q,P^\vee)\buildrel{\cong}\over\rightarrow \text{Tor}^*_B(P,Q)^\vee
$$
for any differential graded algebra $B$, right $B$-module $P$ and left $B$-module $Q$,
we see that the dual of $\Phi$ is an isomorphism of algebras with
respect to the loop product and to the multiplication defined in Theorem~\ref{thm:freeloopExt}.
\end{proof}

\section{Associativity of the loop product on a Poincar\'e duality
  space}\label{associativity loop product}
In this section, by applying the same argument as in the proof of \cite[Theorem 2.2]{T}, 
we shall prove the associativity of the loop products. 

\medskip
\noindent
{\it Proof of Proposition \ref{prop:loop_homology}}. We prove the proposition in the case where $N=M$. 
The same argument as in the proof permits us to conclude that the loop homology ${\mathbb H}_*(L_NM)$ is associative with respect to 
the relative loop products.

Let $M$ be a simply-connected Gorenstein space of dimension $d$. 
In order to prove the associativity of the dual to $Dlp$, we first consider the diagram  
$$
\xymatrix@C25pt@R15pt{
LM\times LM & (LM\times_M LM)\times LM \ar[l]_(0.6){Comp\times 1} \ar[r]^{q\times 1} & LM \times LM \times LM \\
LM\times_M LM \ar[u]^q \ar[d]_{Comp} & LM\times_M LM \times_M LM \ar[l]_(0.6){Comp\times_M 1} \ar[r]^{q\times_M1} \ar[u]_{1\times_ Mq} 
\ar[d]^{1\times_M Comp} &  
LM \times (LM \times_M LM) \ar[u]_{1\times q} \ar[d]^{1\times Comp}\\
LM & LM\times_M LM \ar[l]^{Comp} \ar[r]_q & LM\times LM, 
}
$$
for which the lower left hand-side square is homotopy commutative and other three square 
are strictly commutative.
%
Consider the corresponding diagram
$$
\xymatrix@C20pt@R15pt{
H^*(LM\times LM)\ar[r]^(0.4){(Comp\times 1)^*}
& H^*((LM\times_M LM)\times LM)  \ar[r]^{\varepsilon' \alpha H(q^!)\otimes 1}
&H^*(LM \times LM \times LM) \\
H^*(LM\times_M LM) \ar[u]^{H(q^!)} \ar[r]_(0.4){(Comp\times_M 1)^*}
& H^*(LM\times_M LM \times_M LM) \ar[r]^{H((q\times_M1)^!)} \ar[u]_{H((1\times_ Mq)^!)}
 &  H^*(LM \times (LM \times_M LM)) \ar[u]_{\varepsilon\alpha'1\otimes H(q^!)}\\
H^*(LM) \ar[u]^{Comp^*} \ar[r]_{Comp^*}
& H^*(LM\times_M LM) \ar[r]_{H(q^!)} \ar[u]^{(1\times_M Comp)^*}
& H^*(LM\times LM)\ar[u]_{(1\times Comp)^*}. 
}
$$
The lower left square commutes obviously.
By Theorem~\ref{naturalityshriek}, the upper left square and the lower right square are commutative.
We now show that the upper right square commutes.

By Theorem~\ref{produitshriek}, we see that $H((q\times 1)^!)=\alpha H(q^!)\otimes 1$
and $H((1\times q)^!)=\alpha'1\otimes H(q^!)$ where $\alpha$ and $\alpha'\in \K^*$.
By virtue of \cite[Theorem C]{F-T}, in  $\D(\text{Mod-}C^*(LM^{\times 3}))$, 
$$\varepsilon'(\Delta \times 1)^!\circ \Delta^! =\varepsilon(1\times \Delta)^!\circ \Delta^!$$
where $(\varepsilon,\varepsilon')\neq (0,0)\in \K\times \K$.
Therefore the uniqueness of the shriek map implies that 
$$
\varepsilon'(q\times 1)^!\circ (1 \times_M q)^! =\varepsilon(1\times q)^!\circ (q\times_M 1)^!
$$ 
in $\D(\text{Mod-}C^*(LM^{\times 3}))$; see \cite[Theorem 13]{F-T}.  

So finally, we have proved that
$$\varepsilon'\alpha(Dlp \otimes 1)\circ  Dlp=\varepsilon\alpha'(1\otimes Dlp)\circ  Dlp.$$

Suppose that $M$ is a Poincar\'e duality space of dimension $d$.
By part (2) of Theorem~\ref{produitshriek},
$\alpha=1$ and $\alpha'=(-1)^d$.
Since $\varepsilon (\omega_M\times \omega_M\times\omega_M)=\varepsilon H((\Delta \times 1)^!)\circ H(\Delta^!)  (\omega_M)=
\varepsilon' H((1\times\Delta)^!)\circ H(\Delta^!)  (\omega_M)=\varepsilon' (\omega_M\times \omega_M\times\omega_M)$, we see that
$\varepsilon=\varepsilon'$. Therefore $(Dlp \otimes 1)\circ
Dlp=(-1)^d(1\otimes Dlp)\circ  Dlp$. Thus Lemma
\ref{lem:important_degree}(i)
together with Lemma~\ref{lem:dualsign} (i) and (ii)
 yields that the product 
	$m : {\mathbb H}_*(LM)\otimes {\mathbb H}_*(LM) \to {\mathbb H}_*(LM)$ is  associative.

We prove that the loop product is graded commutative. Consider the commutative diagram 
$$
\xymatrix@C10pt@R5pt{
  & LM \times_M LM \ar[ld]_(0.5){T} \ar@{->}'[d][dd]  \ar[rr]^{q} & & LM\times LM  
  \ar[ld]^{T} \ar[dd]^{p\times p} \\
LM \times_M LM\ar[rr]^(0.6){q} \ar[dd]_p & & LM \times LM  \ar[dd]^(0.3){p\times p} & \\
   & M \ar@{->}'[r]^(0.6){\Delta}[rr] \ar@{=}[ld]& & M \times M \ar[ld]_{T}\\
M \ar[rr]_\Delta  & & M\times M. &
}
$$
By Theorem \ref{thm:newThm3} below,
$H(q^!)\circ T^*=\varepsilon T^*\circ H(q^!)$.
Since $Comp\circ T$ is homotopic to $Comp$,
$Dlp=\varepsilon T^*\circ Dlp$.
If $M$ is a Poincar\'e duality space with orientation class $\omega_M\in H^d(M)$
then $T^*(\omega_M\otimes \omega_M) =(-1)^{d^2} (\omega_M\otimes \omega_M)$.
Therefore by part a) of Remark~\ref{remafternewthm3}, $\varepsilon=(-1)^d$. 
By Lemma  \ref{lem:important_degree}(ii) together with Lemma~\ref{lem:dualsign}(i), we see that the product $m$ is graded commutative. 
This completes the proof. 
\hfill\qed

\begin{rem}
The commutativity of the loop homology ${\mathbb H}_*(L_NM)$ 
does not follow from the proof of Proposition \ref{prop:loop_homology}. In general $Comp\circ T$ is not homotopic to 
$Comp$ in $L_NM$.  As mentioned in the Introduction, the relative loop product is not necessarily commutative; see \cite{Naito}.
\end{rem}


\section{Appendix: Properties of shriek maps}
In this section, we extend the definitions and properties of shriek
maps on Gorenstein spaces given in~\cite{F-T}. These properties are
used in section~\ref{associativity loop product}.
\begin{defn}
A pull-back diagram,
$$
\xymatrix@C15pt@R15pt{
 X   \ar[d]_{q}  \ar[rr]^g &&E \ar[d]^p\\
 N\ar[rr]^f &&M}
$$
satisfies Hypothesis (H) (Compare with the hypothesis (H) described in \cite[page 418]{F-T}) if $p:E\twoheadrightarrow M$ is a fibration,
for any $n\in\mathbb{N}$, $H^n(E)$ is of finite dimension 
and 

\smallskip
\centerline{$
\hspace{5mm}\left\{
\begin{array}{l}
  N \mbox{ is an oriented  Poincar\'e duality space of dimension } n\,,\\
 M  \mbox{ is a  1-connected   oriented  Poincar\'e duality space of dimension } m\,,
\end{array}\right.
$}

\smallskip
\noindent or $f:B^r\rightarrow B^t$ is the product of diagonal maps
$B\rightarrow B^{n_i}$,
the identity map of $B$, the inclusion $\eta:{*}\rightarrow B$ for a simply-connected
$\K$-Gorenstein space $B$.
\end{defn}
Let $n$ be the dimension of $N$ or $r$ times the dimension of $B$.
Let $m$ be the dimension of $M$ or $t$ times the dimension of $B$.
It follows from \cite[Lemma 1 and Corollary p. 448]{F-T}
that $H^q(N)\cong \text{Ext}_{C^*(M)}^{q+m-n}(C^*(N),C^*(M))$.
By definition, a shriek map $f^!$ for $f$ is a generator of $\text{Ext}_{C^*(M)}^{\leq
  m-n}(C^*(N),C^*(M))$. Moreover, there exists an unique element  $g^!\in
\text{Ext}_{C^*(E)}^{m-n}(C^*(X),C^*(E))$ such that
$g^!\circ C^*(q)= C^*(p)\circ f^!$ 
in the derived category of $C^*(M)$-modules; see Theorem \ref{thm:main_F-T}. 

Here we have extended the definitions of shriek maps due to Felix and
Thomas in order to include the following example and the case $(\Delta\times 1)^!$ that we use in 
the proof of Proposition \ref{prop:loop_homology}.
\begin{ex}\label{intersection fibration}(Compare with~\cite[p. 419-420]{F-T} where $M$ is a Poincar\'e duality space)
Let $F\buildrel{\widetilde{\eta}}\over\rightarrow
E\buildrel{p}\over\rightarrow M$ be a fibration over a
simply-connected Gorenstein space $M$ with generator
$\eta^!=\omega_M\in\text{Ext}^m_{C^*(M)}(\K,C^*(M))$.
By definition, $H(\widetilde{\eta}^!):H^*(F)\rightarrow H^{*+m}(E)$
is the dual to the {\it intersection morphism}.

Let $G$ be a connected Lie group. Then its classifying space $BG$ is an
example of Gorenstein space of negative dimension.
Let $F$ be a $G$-space. It is not difficult to see that our intersection morphism of $F\rightarrow
F\times_G EG\rightarrow BG$  coincides with the integration along
the
fibre of the principal $G$-fibration $G\rightarrow F\times EG\rightarrow
F\times_G EG$ for an appropriate choice of the generator $\eta^!$; see the proof of \cite[Theorem 6]{F-T}. 

Suppose now that $F\buildrel{\widetilde{\eta}}\over\rightarrow
E\buildrel{p}\over\rightarrow M$ is a monoidal fibration.
With the properties of shriek maps given in this section,
generalizing~\cite[Theorem 10]{F-T} (See also~\cite[Proposition 10]{G-S}) in the
Gorenstein case,
one can show that the intersection morphism
$H(\widetilde{\eta}_!):H_{*+m}(E)\rightarrow H_*(F)$
is multiplicative if in the derived category of $C^*(M\times
M)$-modules
$$
\Delta^!\circ \omega_M=\omega_M\times \omega_M.
\eqnlabel{add-1}
$$
The generator $\Delta^!\in\text{Ext}^m_{C^*(M^2)}(C^*(M),C^*(M^2))$
is defined up to a multiplication by a scalar.
If we could prove that $\Delta^!\circ \omega_M$ is always not zero,
we would have an unique choice for $\Delta^!$ satisfying (11.1). 
Then we would have solved the ``up to a constant
problem'' mentioned in~\cite[Q1 p. 423]{F-T}.
\end{ex}

We now describe a generalized version of \cite[Theorem 3]{F-T}. We consider the following commutative diagram.  
$$
\xymatrix@C15pt@R5pt{
  & X  \ar[ld]_{v} \ar@{->}'[d][dd]_(0.4){q}  \ar[rr]^{g} & & E \ar[ld]^(0.4){k} \ar[dd]^{p} \\
X' \ar[rr]^(0.6){g'} \ar[dd]_{q'} & & E' \ar[dd]^(0.4){p'}& \\
   & N\ar@{->}'[r]_(0.7){f}[rr] \ar@{->}[ld]_{u} & & M \ar[ld]^{h} \\
N' \ar[rr]_{f'}  & & M'  &
}
$$
in which the back and the front squares satisfies Hypothesis (H). 

\begin{thm} \label{thm:newThm3} {\em (}Compare with \cite[Theorem 3]{F-T}{\em )}   
With the above notations, suppose that $m'-n'=m-n$.

{\em (1)} If $h$ is a homotopy equivalence then in the derived category of $C^*(M')$-modules, $f^!\circ C^*(u)=\varepsilon C^*(h)\circ f'^!$, 
where $\varepsilon\in \K$.

{\em (2)} If in the derived category of $C^*(M')$-modules, $f^!\circ C^*(u)=\varepsilon C^*(h)\circ f'^!$ then
in the derived category of $C^*(E')$-modules, $g^!\circ C^*(v)=\varepsilon C^*(k)\circ g'^!$. In particular, 
$$
H^*(g^!)\circ H^*(v) = \varepsilon H^*(k)\circ H^*({g'}^!) .
$$
\end{thm}
\begin{rem}\label{remafternewthm3}
a) In (1), if $N'$ and $M'$ are oriented Poincar\'e duality spaces, the constant
$\varepsilon$ is given by
$$
H^n(f^!)\circ H^n(u)(\omega_{N'}) = \varepsilon H^{m'}(h)\circ H^{n'}({f'}^!)(\omega_{N'}) .
$$
In fact, this is extracted from the uniqueness of the shriek map described in \cite[Lemma 1]{F-T}. 

b) In \cite[Theorem 3]{F-T}, it is not useful that $v$ and $k$ are homotopy equivalence. But in \cite[Theorem 3]{F-T}, the homotopy equivalences $u$ and $h$
should be orientation preserving in order to deduce $\varepsilon =1$.

c) If the bottom square is the pull-back along a smooth embedding $f'$
of compact oriented manifolds and a smooth map $h$ transverse to $N'$.
Then by~\cite[Proposition 4.2]{Me}, $f^!\circ C^*(u)=
C^*(h)\circ f'^!$ and $
H^*(g^!)\circ H^*(v) = H^*(k)\circ H^*({g'}^!)$.
\end{rem}
\begin{proof}[Proof of Theorem \ref{thm:newThm3}]
The proofs of (1) and (2) follow from the proof of \cite[Theorem 3]{F-T}.
But we review this proof, in order to explain that
Theorem~\ref{thm:newThm3} is valid in the Gorenstein case
and that we don't need to assume as in~\cite[Theorem 3]{F-T}
that $u$, $k$ and $v$ are homotopy equivalence.

(1) Since $h$ is a homotopy equivalence,
\begin{eqnarray*}
\text{Ext}^*_{C^*(M')}(C^*(N'),C^*(h)): &&\\  
&& \hspace{-3cm} \text{Ext}^*_{C^*(M')}(C^*(N'),C^*(M')) \rightarrow 
\text{Ext}^*_{C^*(M')}(C^*(N'),C^*(M))
\end{eqnarray*}
is an isomorphism.
By definition~\cite[Theorem 1 and p. 449]{F-T}, the shriek map $f'^!$ is a generator
of
$\text{Ext}^{m'-n'}_{C^*(M')}(C^*(N'),C^*(M'))\cong \K$. Then 
$C^*(h)\circ f'^!$ is a generator of $\text{Ext}^{m'-n'}_{C^*(M')}(C^*(N'),C^*(M))$.
So since $f^!\circ C^*(u)$ is in $\text{Ext}^{m-n}_{C^*(M')}(C^*(N'),C^*(M))$, we have (1).

(2) Let $P$ be any $C^*(E')$-module. Since $X'$ is a pull-back, a straightforward generalization of~\cite[Theorem 2]{F-T}
shows that
$$
\text{Ext}^*_{C^*(p')}(C^*(q'),P):
\text{Ext}^*_{C^*(E')}(C^*(X'),P)\rightarrow 
\text{Ext}^*_{C^*(M')}(C^*(N'),P)
$$ is an isomorphism.
Take $P:=C^*(E)$.
Consider the following cube in the derived category of $C^*(M')$-modules.
$$
\xymatrix@C15pt@R8pt{
  & C^*(X)    \ar[rr]^{g^!}
  & & C^*(E)  \\
C^*(X') \ar[ru]^{C^*(v)} \ar[rr]^(0.6){g'^!} & & C^*(E') \ar[ur]_(0.4){C^*(k)} & \\
   & C^*(N) \ar@{->}'[u][uu]^(0.4){C^*(q)}   \ar@{->}'[r]_(0.7){f^!}[rr]  & & C^*(M). \ar[uu]_{C^*(p)} \\
C^*(N') \ar[rr]_{f'^!} \ar[uu]^{C^*(q')} \ar@{->}[ur]^{C^*(u)} & & C^*(M') \ar[ur]_{C^*(h)} \ar[uu]_(0.4){C^*(p')}&
}
$$
Since in $\text{Ext}^*_{C^*(M')}(C^*(N'),C^*(E))$, the elements
$g^!\circ C^*(v)\circ C^*(q')$ and $\varepsilon C^*(k)\circ g'^!\circ C^*(q')$ are equal, the assertion (2) follows.
\end{proof}

When $u$ and $h$ are the identity maps, Theorem~\ref{thm:newThm3}
gives~\cite[Theorem 4]{F-T} (Compare with~\cite[Lemma 4]{G-S})
and the following variant for Gorenstein spaces: 

\begin{thm}{\em (}Naturality of shriek maps with respect to pull-backs {\em )}\label{naturalityshriek}
Consider the two pull-back squares
$$
\xymatrix@C20pt@R10pt{
X\ar[r]^{g}\ar[d]_v
& E\ar[d]^{k}\\
X'\ar[r]^{g'}\ar[d]_{q'}
& E'\ar[d]^{p'}\\
B^r\ar[r]_{\Delta}
& B^t
}
$$
where $\Delta:B^r\rightarrow B^t$ is the product of diagonal maps of a simply-connected
$\K$-Gorenstein space $B$ and $p'$ and $p'\circ k$ are two fibrations.
Then in the derived category of $C^*(E')$-modules,
$g^!\circ C^*(v)= C^*(k)\circ g'^!.$
\end{thm}

 \begin{thm}{\em (}Products of shriek maps {\em )}\label{produitshriek}
Let 

\  \  \  \  \  \  \
$
\xymatrix@C15pt@R15pt{
 X   \ar[d]_{q}  \ar[rr]^g &&E \ar[d]^p&\ar@{}[d] & \\
 N\ar[rr]^f && M&}
$ 
and \ \  \  \  \  \  \  \  \  \  \  \  
$\xymatrix@C15pt@R15pt{
X'   \ar[d]_{q'}  \ar[rr]^{g'} &&E' \ar[d]^{p'}\\
 N'\ar[rr]^{f'} &&M'}
$

\medskip
\noindent
be two pull-back diagrams satisfying Hypothesis (H).
Let $$EZ^\vee:C^*(M\times M')\buildrel{\simeq}\over\rightarrow \left(C_*(M)\otimes C_*(M')\right)^\vee$$
be the quasi-isomorphism of algebras dual to the Eilenberg-Zilber morphism.
Let $$\Theta: C^*(M)\otimes C^*(M')\buildrel{\simeq}\over\rightarrow\left(C_*(M)\otimes C_*(M')\right)^\vee$$ be the quasi-isomorphism of algebras
sending the tensor product of cochains $\varphi\otimes \varphi'$ to the form denoted again $\varphi\otimes \varphi'$ defined
by $(\varphi\otimes \varphi')(a\otimes b)=(-1)^{\vert\varphi'\vert\vert a\vert}\varphi(a)\varphi'(b)$.
Then

{\em (1)} there exists $h\in\text{\em Ext}^{m+m'-n-n'}_{\left(C_*(M)\otimes C_*(M')\right)^\vee}(\left(C_*(N)\otimes C_*(N')\right)^\vee,\left(C_*(M)\otimes C_*(M')\right)^\vee )$ such that in the derived category of $C^*(M\times M')$-modules
$$
\xymatrix@C15pt@R15pt{
C^*(N\times N')\ar[r]^ {(f\times f')^!}\ar[d]_{EZ^\vee}
& C^*(M\times M')\ar[d]^{EZ^\vee}\\
\left(C_*(N)\otimes C_*(N')\right)^\vee\ar[r]_{h}
&\left(C_*(M)\otimes C_*(M')\right)^\vee
}
$$
and in the derived category of $C^*(M)\otimes C^*(M')$-modules
$$
\xymatrix@C15pt@R15pt{
\left(C_*(N)\otimes C_*(N')\right)^\vee\ar[r]^{h}
&\left(C_*(M)\otimes C_*(M')\right)^\vee\\
C^*(N)\otimes C^*(N')\ar[r]_{\varepsilon f^!\otimes f'^!}\ar[u]^{\Theta}
& C^*(M)\otimes C^*(M')\ar[u]_{\Theta}
}
$$
are commutative squares for some $\varepsilon\in \K^*$.

{\em (2)}  Suppose that $N$, $N'$, $M$ and $M'$ are Poincar\'e duality spaces oriented by $\omega_N\in H^n(N)$, $\omega_{N'}\in H^{n'}(N')$, $\omega_M\in H^m(M)$ and $\omega_{M'}\in H^{m'}(M')$. If we orient $N\times N'$ by $\omega_{N}\times\omega_{N'}$ and $M\times M'$ by $\omega_{M}\times\omega_{M'}$ then
$\varepsilon=(-1)^{(m'-n')n}$.

{\em (3)} There exists $k\in\text{\em Ext}^{m+m'-n-n'}_{\left(C_*(E)\otimes C_*(E')\right)^\vee}(\left(C_*(X)\otimes C_*(X')\right)^\vee,\left(C_*(E)\otimes C_*(E')\right)^\vee )$ such that in the derived category of $C^*(E\times E')$-modules
$$
\xymatrix@C15pt@R15pt{
C^*(X\times X')\ar[r]^ {(g\times g')^!}\ar[d]_{EZ^\vee}
& C^*(E\times E')\ar[d]^{EZ^\vee}\\
\left(C_*(X)\otimes C_*(X')\right)^\vee\ar[r]_{k}
&\left(C_*(E)\otimes C_*(E')\right)^\vee
}
$$
and in the derived category of $C^*(E)\otimes C^*(E')$-modules
$$
\xymatrix@C15pt@R15pt{
\left(C_*(X)\otimes C_*(X')\right)^\vee\ar[r]^{k}
&\left(C_*(E)\otimes C_*(E')\right)^\vee\\
C^*(X)\otimes C^*(X')\ar[r]_{\varepsilon g^!\otimes g'^!}\ar[u]^{\Theta}
& C^*(E)\otimes C^*(E')\ar[u]_{\Theta}
}
$$
are commutative squares.  
\end{thm}

\begin{rem}
%
Here $g^!\otimes g'^!$ denotes the $C^*(E)\otimes C^*(E')$-linear map defined
by
$$
(g^!\otimes g'^!)(a\otimes b)=(-1)^{\vert g'^!\vert\vert a\vert}g^!(a)\otimes g'^!(b).
$$
Therefore, Theorem~\ref{produitshriek} (2)  implies that
$$
H^*((g\times g')^!)(a\times b)=(-1)^{(m'-n')(n+\vert a\vert)}H^*(g^!)(a)\times H^*(g'^!)(b).
$$
The signs of~\cite[VI.14.3]{Bredongeometry} are different from that mentioned here.
\end{rem}

\begin{proof}[Proof of Theorem \ref{produitshriek}]
(1) By definition~\cite[Theorem 1 and p. 449]{F-T}, $(f\times f')^!$ is a generator of $\text{Ext}^{\leq m+m'-n-n'}_{C^*(M \times M')}(C^*(N\times N'),
C^*(M \times M'))$.
Let $h$ be the image of $(f\times f')^!$ by the composite of isomorphisms
$$\xymatrix@C15pt@R17pt{
\text{Ext}^{*}_{C^*(M \times M')}(C^*(N\times N'),
C^*(M \times M'))\ar[d]_{\text{Ext}^{*}_{Id}(Id, EZ^\vee)}^\cong\\
\text{Ext}^{*}_{C^*(M \times M')}(C^*(N\times N'),
\left(C_*(M)\otimes C_*(M')\right)^\vee)\\
\text{Ext}^{*}_{\left(C_*(M)\otimes C_*(M')\right)^\vee}(\left(C_*(N)\otimes C_*(N')\right)^\vee,\left(C_*(M)\otimes C_*(M')\right)^\vee ).
\ar[u]^{\text{Ext}^{*}_{EZ^\vee}(EZ^\vee,Id)}_\cong
}
$$
Since $f^!\otimes f'^!$ is a generator of
$$\text{Ext}^{\leq m+m'-n-n'}_{C^*(M)\otimes C^*(M')}(C^*(N)\otimes C^*(N'),C^*(M)\otimes C^*(M'))$$
$$\cong \text{Ext}^{\leq m-n}_{C^*(M)}(C^*(N),C^*(M))\otimes
\text{Ext}^{\leq m'-n'}_{C^*(M')}(C^*(N'),C^*(M')),$$
the image of $h$ by the composite of isomorphisms 
$$\xymatrix@C15pt@R17pt{
\text{Ext}^{*}_{\left(C_*(M)\otimes C_*(M')\right)^\vee}(\left(C_*(N)\otimes C_*(N')\right)^\vee,\left(C_*(M)\otimes C_*(M')\right)^\vee )
\ar[d]^{\text{Ext}^{*}_{\Theta}(\Theta,Id)}_\cong\\
\text{Ext}^{*}_{C^*(M)\otimes C^*(M')}(C^*(N)\otimes C^*(N'),
\left(C_*(M)\otimes C_*(M')\right)^\vee)\\
\text{Ext}^{*}_{C^*(M)\otimes C^*(M')}(C^*(N)\otimes C^*(N'),C^*(M)\otimes C^*(M')).
\ar[u]_{\text{Ext}^{*}_{Id}(Id, \Theta)}^\cong
}
$$
is an element $\varepsilon(f^!\otimes f'^!)$, where $\varepsilon$ is a non-zero constant
(2) In cohomology, (1) gives a commutative diagram 
$$
\xymatrix@C15pt@R25pt{
H^*(N\times N')\ar[r]^{H^*((f\times f')^!)} &H^*(M\times M')\\
H^*(N)\otimes H^*(N') \ar[u]^{\times}\ar[r]_{\varepsilon H^*(f^!)\otimes H^*(f'^!)}
&H^*(M)\otimes H^*(M'),  \ar[u]_{\times}
}
$$
where $\times$ is the cross product.
Therefore
\begin{multline*}
\omega_M\times\omega_{M'}=H^*((f\times f')^!)(\omega_N\times\omega_{N'})=\\
\varepsilon(-1)^{(m'-n')n}H^*(f^!)(\omega_{N})\times H^*(f'^!)(\omega_{N'})=\varepsilon(-1)^{(m'-n')n}\omega_M\times\omega_{M'}.
\end{multline*}

\noindent  (3) Consider the following cube in the derived category of $C^*(M\times M')$-modules
$$
{\footnotesize 
\xymatrix@C5pt@R15pt{
  & \left(C_*(X)\otimes C_*(X')\right)^\vee    \ar[rr]^{k}
  & & \left(C_*(E)\otimes C_*(E')\right)^\vee  \\
C^*(X\times X') \ar[ru]^{EZ^\vee} \ar[rr]_(0.6){(g\times g')^!} & & C^*(E\times E') \ar[ur]_(0.4){EZ^\vee} & \\
   & \left(C_*(N)\otimes C_*(N')\right)^\vee \ar@{->}'[u][uu]_(0.4){\left(C_*(q)\otimes C_*(q')\right)^\vee}   \ar@{->}'[r]_(0.7){h}[rr] 
   & & \left(C_*(M)\otimes C_*(M')\right)^\vee \ar[uu]_{\left(C_*(p)\otimes C_*(p')\right)^\vee} \\
C^*(N\times N') \ar[rr]_{(f\times f')^!} \ar[uu]^{C^*(q\times q')} \ar@{->}[ur]^{EZ^\vee} & & C^*(M\times M') \ar[ur]_{EZ^\vee} \ar[uu]_(0.4){C^*(p\times p')}&
}
}
$$
with $k$ defined below. 
Since
$${\text{Ext}^{*}_{C^*(p \times p')}(C^*(q\times q'),C^*(E \times E'))}:$$
$$
\text{Ext}^{*}_{C^*(E \times E')}(C^*(X\times X'),
C^*(E \times E'))\rightarrow
\text{Ext}^{*}_{C^*(M \times M')}(C^*(N\times N'),
C^*(E \times E'))
$$
is an isomorphism, it follows that the maps 
$$
\text{Ext}^{*}_{C^*(p \times p')}(C^*(q\times q'),
\left(C_*(E)\otimes C_*(E')\right)^\vee):$$
\begin{eqnarray*}
\text{Ext}^{*}_{C^*(E \times E')}(C^*(X\times X'),
\left(C_*(E)\otimes C_*(E')\right)^\vee) & \\
 & \hspace{-3cm}\rightarrow 
\text{Ext}^{*}_{C^*(M \times M')}(C^*(N\times N'),
\left(C_*(E)\otimes C_*(E')\right)^\vee)
\end{eqnarray*}
and
$$\text{Ext}^{*}_{\left(C_*(p)\otimes C_*(p')\right)^\vee}(\left(C_*(q)\otimes C_*(q')\right)^\vee,\left(C_*(E)\otimes C_*(E')\right)^\vee ):$$
\begin{align*}
\text{Ext}^{*}_{\left(C_*(E)\otimes C_*(E')\right)^\vee}(\left(C_*(X)\otimes C_*(X')\right)^\vee, 
\left(C_*(E)\otimes C_*(E')\right)^\vee ) \\
\rightarrow
\text{Ext}^{*}_{\left(C_*(M)\otimes C_*(M')\right)^\vee}(\left(C_*(N)\otimes C_*(N')\right)^\vee, 
\left(C_*(E)\otimes C_*(E')\right)^\vee )
\end{align*}
are also isomorphisms.
Let $k$ be the image of $\left(C_*(p)\otimes C_*(p')\right)^\vee\circ h$
by the inverse of the isomorphism 
$$\text{Ext}^{*}_{\left(C_*(p)\otimes C_*(p')\right)^\vee}(\left(C_*(q)\otimes C_*(q')\right)^\vee,\left(C_*(E)\otimes C_*(E')\right)^\vee ).$$

Since $EZ^\vee\circ (g\times g')^!$ and $k\circ EZ^\vee$ have the same image by
$$
\text{Ext}^{*}_{C^*(p \times p')}(C^*(q\times q'),
\left(C_*(E)\otimes C_*(E')\right)^\vee),$$
they coincide and hence we have proved the commutativity of the first square in (3).
For the second square in (3), the proof is the same using this time the following cube
in the derived category of $C^*(M)\otimes C^*(M')$-modules
$$
{\footnotesize
\xymatrix@C10pt@R15pt{
  & \left(C_*(X)\otimes C_*(X')\right)^\vee    \ar[rr]^{k}
  & & \left(C_*(E)\otimes C_*(E')\right)^\vee  \\
C^*(X)\otimes C^*(X') \ar[ru]^{\Theta} \ar[rr]_(0.6){\varepsilon g^!\otimes g'^!} & & C^*(E)\otimes C^*(E') \ar[ur]_-{\Theta} & \\
   & \left(C_*(N)\otimes C_*(N')\right)^\vee \ar@{->}'[u][uu]_(0.4){\left(C_*(q)\otimes C_*(q')\right)^\vee}   \ar@{->}'[r]_(0.7){h}[rr] 
   & & \left(C_*(M)\otimes C_*(M')\right)^\vee \ar[uu]_{\left(C_*(p)\otimes C_*(p')\right)^\vee} \\
C^*(N)\otimes C^*(N') \ar[rr]_{\varepsilon f^!\otimes f'^!} \ar[uu]_(0.7){C^*(q)\otimes C^*(q')} \ar@{->}[ur]^{\Theta} & & C^*(M)\otimes C^*(M') \ar[ur]_{\Theta} \ar[uu]_(0.4){C^*(p)\otimes C^*(p')}&
}
}
$$
\end{proof}

\section{Appendix: shriek maps and Poincar\'e duality}
In this section, we compare precisely the shriek map defined by Felix and Thomas in ~\cite[Theorem A]{F-T}
with various shriek maps defined by Poincar\'e duality. 

We first remark the cap products used in the body of the present paper. 
Let  $\text{--} \cap \sigma : C^{*}(M)\to C_{m-*}(M)$ be 
the cap product given in ~\cite[VI.5]{Bredongeometry}, where $\sigma \in C_{m}(M)$. 
Observe that $\text{--}\cap \sigma$ is defined by 
$$(\text{--} \cap \sigma)(x) = (-1)^{\vert m \vert\vert x\vert} x \cap \sigma.$$ 
By \cite[VI.5.1 Proposition (iii)]{Bredongeometry}, the map $-\cap \sigma $ 
is a morphism of left $C^{*}(M)$-modules. 
The sign $(-1)^{\vert m \vert\vert x\vert}$ makes $\text{--} \cap \sigma$ a left $C^*(M)$ linear map
in the sense that $f(xm)=(-1)^{\vert f \vert \vert x\vert} xf(m)$
as quoted in \cite[p. 44]{F-H-T}. 

We denote by $\sigma \cap \text{ --} : C^{*}(M)\to C_{m-*}$ the cap 
product described in \cite[\S 7]{Menichi}. The map $\sigma \cap -$ is 
a right $C^{*}(M)$-module map (\cite[Proposition 2.1.1]{Sweedler}). Moreover,  
we see that $x\cap \sigma =(-1)^{m|x|}\sigma \cap x $ in homology for any $x\in H^{*}(M)$.

\begin{prop}\label{shriek and Poincare duality}
Let $N$ and $M$ be two oriented Poincar\'e duality space of dimensions $n$ and $m$.
Let $[N]\in H_n(N)$, $[M]\in H_n(M)$ and $\omega_N\in H^n(N)$, $\omega_M\in H^m(M)$
such that the Kronecker products $<\omega_M,[M]>=1=<\omega_N,[N]>$.
Let $f:N\rightarrow M$ be a continous map.
Let $f^!$ be the unique element of $\text{\em Ext}^{m-n}_{C^*(M)}(C^*(N),C^*(M))$
such that $H(f^!)(\omega_N)=\omega_M$. Then 

1) The diagram in the derived category of right-$C^*(M)$ modules
$$
\xymatrix@C20pt@R20pt{
C^*(N)\ar[r]^{f^!}\ar[d]_{[N]\cap -}
& C^{*+m-n}(M)\ar[d]^{[M]\cap -}\\
C_{n-*}(N)\ar[r]_{C_*(f)}
& C_{n-*}(M)
}
$$
commutes up to the sign $(-1)^{m+n}$.

2) Let $\psi_N:C_*(N)\hookrightarrow C^*(N)^\vee$ be the canonical inclusion of the complex $C_*(N)$
into its bidual.
The diagram of left $H^*(M)$-modules 
$$
\xymatrix@C20pt@R15pt{
H^*(M)^\vee\ar[r]^{H^*(f^!)^\vee}
&H^{*+n-m}(N)^\vee\\
H_*(M)\ar[u]^{H(\psi_M)}
& H_{*+n-m}(N)\ar[u]_{H(\psi_N)}\\
H^{m-*}(M)\ar[r]_{H^{*}(f)}\ar[u]^{-\cap [M]}
&H^{m-*}(N)\ar[u]_{-\cap [N]}
}
$$
commutes up to the sign $(-1)^{n(m-n)}$.

3) Let $\int_N:H(A(N))\buildrel{\cong}\over\rightarrow H(C^*(N))$ be the isomorphism of algebras
induced by the natural equivalence between the rational singular cochains and the polynomial differential forms~\cite[Corollary 10.10]{F-H-T}.
Let $\theta_N:A(N)\buildrel{\simeq}\over\rightarrow A(N)^\vee$ be a morphism of $A(N)$-modules
such that the class of $\theta_N(1)$ is the fundamental class of $N$, $[N]\in H_n(A(N)^\vee)\cong
H_n(N;\mathbb{Q})$.
Let $f^!_A$ be the unique element of $\text{\em Ext}^{m-n}_{A(M)}(A(N),A(M))$
such that $H(f^!_A)(\int_N^{-1}\omega_N)=\int_M^{-1}\omega_M$.
Then in the derived category of $A(M)$-modules, the diagram
$$
\xymatrix@C20pt@R20pt{
A(M)^\vee\ar[r]^{(f^!_A)^\vee}
&A(N)^\vee\\
A(M)\ar[r]_{A(f)}\ar[u]^{\theta_M}
&A(N)\ar[u]_{\theta_N}
}
$$
commutes also with the sign $(-1)^{n(m-n)}$.
\end{prop}
\begin{rem}
Part 1) of the previous proposition is already in~\cite[p. 419]{F-T} but without sign and with left-$C^*(M)$ modules. In particular, they should have defined their maps $\text{--} \cap [M]:C^*(M)\rightarrow 
C_{m-*}(M)$ by $(\text{--} \cap [M])(x)=(-1)^{\vert x\vert m}  x\cap [M]$ in order to have a left-$C^*(M)$ linear
map~\cite[p. 283]{Friedman}. 

Note that diagram of part 2) of Proposition \ref{shriek and Poincare duality} is commutative at 
the cochain level as our proof below shows it.
\end{rem}

\begin{proof}
1) By~\cite[Lemma 1]{F-T}, it suffices to show that the diagram commutes in homology on the generator
$\omega_M\in H^m(M)$. Let $\varepsilon_M:H_0(M)\buildrel{\cong}\over\rightarrow\K$ be the augmentation.
It is well-known that $\varepsilon_M(\omega_M\cap [M])=<\omega_M,[M]>=1$.
Therefore  $\varepsilon_M([M]\cap \omega_M)=(-1)^m\varepsilon_M(\omega_M\cap [M])=(-1)^m$.
On the other hand, $\varepsilon_N([N]\cap H(f^!)(\omega_M))=\varepsilon_N([N]\cap \omega_N)=(-1)^n$.

2) Since $H^*(f^!)$ is right $H^*(M)$-linear, its dual is left $H^*(M)$-linear. 
By $H^*(M)$-linearity, it suffices to check that the diagram commutes on $1\in H^*(M)$.
By~\cite[7.1]{F-T-VP} or~\cite[Property 57 i)]{Menichi}, $\psi_M:C_*(M)\hookrightarrow C^*(M)^\vee$
is defined by $\psi_M(c)(\varphi)=(-1)^{\vert c\vert\vert\varphi\vert}\varphi(c)$
for $c\in C_*(M)$ and $\varphi\in C^*(M)$.
Therefore $\psi_N([N])(\omega_N)=(-1)^{n^2}$. And
\begin{align*}
\left((f^!)^\vee\circ\psi_M\right)([M])(\omega_N)=(-1)^{m\vert f^!\vert}\left(\psi_M([M])\circ f^!\right)(\omega_N)\\
=(-1)^{m^2-mn}\psi_M([M])(\omega_M)=(-1)^{mn}.
\end{align*}
So $((f^!)^\vee\circ\psi_M)([M])=(-1)^{n(m-n)}\psi_N([N])$. 
Observe that the map $\psi_M$ is left $C^*(M)$-linear; see \cite[p. 250]{F-T-VP}. 

3) The isomorphism
$ \text{Ext}^{q}_{A(M)}(A(M),A(N)^\vee)\cong H_{-q}(A(N)^\vee)$
maps any element $\varphi$ to $ H(\varphi)(1)$.
Therefore it suffices to check that the diagram commutes in homology on $1\in H^0(A(M))$.
Consider the cube
$$
{\footnotesize
\xymatrix@C20pt@R15pt{
  & H(A(M)^\vee)    \ar[rr]^{H((f^!_A)^\vee)}
  & & H(A(N)^\vee)  \\
H^*(M)^\vee \ar[ru]^{\int_M^\vee} \ar[rr]_(0.6){H((f^!)^\vee)} & & H^*(N)^\vee \ar[ur]_-{\int_N^\vee} & \\
   & H(A(M)) \ar@{->}'[u][uu]_(0.4){H(\theta_M)}   \ar@{->}'[r]_(0.7){H(A(f))}[rr] \ar@{->}[dl]_{\int_M}
   & & H(A(N)) \ar[uu]_{H(\theta_N)} \ar@{->}[dl]^{\int_N}\\
H^*(M) \ar[rr]_{H^*(f)} \ar[uu]^{H(\psi_M)\circ [M]\cap -}  & & 
H^*(N) \ar[uu]_(0.7){H(\psi_N)\circ [N]\cap -}&
}
}
$$
The bottom face commutes by naturality of the isomorphism $\int_N$.
The left face and right faces commute on $1$ by definition of $\theta_M(1)$ and $\theta_N(1)$.
The top face is the dual of the following square
$$
\xymatrix@C35pt@R20pt{
H(A(N))\ar[r]^{H^*(f^!_A)}\ar[d]_{\int_N}
& H(A(M))\ar[d]^{\int_M}\\
H(C^ *(N))\ar[r]_{H^*(f^!)}
& H(C^*(M))\ar[r]_-{H(\psi_M([M]))}
& \K
}
$$
Since $H(\psi_M([M]))\circ H^*(f^!)\circ \int_N=H(\psi_M([M]))\circ \int_M\circ H^*(f^!_A)$, by
Lemma~\ref{caracterisation shriek map cohomology} below, the square is commutative.
Therefore the top face of the previous cube is commutative.

The front face is exactly the diagram in part 2) of this proposition.
Therefore the back face commutes on $1$ up to the same sign $(-1)^{n(m-n)}$.

\end{proof}
The following lemma is a cohomological version of~\cite[Lemma 1]{F-T}.

\begin{lem}\label{caracterisation shriek map cohomology}
Let $P$ be a right $H^*(M)$-module.
Then the map $$\Psi:\text{\em Hom}_{H^*(M)}^q(P,H^*(M))\rightarrow (P^{q-m})^\vee$$
mapping $\varphi$ to the composite $H(\psi_M([M]))\circ\varphi$ is an isomorphism.
\end{lem}

\begin{rem}
This lemma holds also with $\text{Ext}$ instead of $\text{Hom}$ and with $H_*(M)$, $C_*(M)$
or~\cite[Lemma 1]{F-T} $C^*(M)$ instead of $H^*(M)$.
\end{rem}

We give the proof of Lemma~\ref{caracterisation shriek map cohomology}, since we don't understand all the proof
of~\cite[Lemma 1]{F-T}.
\begin{proof}
Since the form $H(\psi_M([M])):H(C^*(M))\rightarrow\K$, $\alpha\mapsto (-1)^{m\vert\alpha\vert}<\alpha,[M]>$ coincides with
$
H^m(C^*(M))\buildrel{[M]\cap -}\over\rightarrow H_0(C^*(M))\buildrel{\varepsilon}\over\cong \K
$,
the map $\Psi$ coincides with the composite of
$$\text{Hom}_{H^*(M)}^q(P,H(\psi_M)\circ [M]\cap -):
\text{Hom}_{H^*(M)}^q(P,H^*(M))\buildrel{\cong}\over\rightarrow
\text{Hom}_{H^*(M)}^{q-m}(P,H^*(M)^\vee)$$
and
$$
\text{Hom}_{H^*(M)}(P,\text{Hom}(H^*(M),\K))\cong
\text{Hom}_\K(P\otimes_{H^*(M)} H^*(M),\K)\cong \text{Hom}_\K(P,\K).
$$
\end{proof}

\begin{cor}\label{shriek diagonal}
Let $N$ be an oriented Poincar\'e duality space.
Let $[N]\in H_n(N)$ and $\omega_N\in H^n(N)$
such that the Kronecker product $<\omega_N,[N]>=1$.
Let $\Delta^!$ be the unique element of $\text{\em Ext}^{n}_{C^*(N\times N)}(C^*(N),C^*(N\times N))$
such that $H(\Delta^!)(\omega_N)=\omega_N\times \omega_N$. 
(This is the $\Delta^!$ considered in all this paper, since we orient $N\times N$ with the cross product
$\omega_N\times \omega_N$: See part (2) of Theorem~\ref{produitshriek}).
Let $\Delta^!_A$ be the unique element of $\text{Ext}^{n}_{A(N\times N)}(A(N),A(N\times N))$
such that $H(\Delta^!_A)(\int_N^{-1}\omega_N)=\int_{N\times N}^{-1}\omega_N\times \omega_N$. 
Then in the case of $\Delta^!$ and of $\Delta^!_A$,
all the diagrams of Proposition~\ref{shriek and Poincare duality}
commute exactly\footnote{As we see in the proof, this is lucky!}.
\end{cor}
\begin{proof}
By~\cite[VI.5.4 Theorem]{Bredongeometry}, the cross products in homology and cohomology and the Kronecker product
satisfy
$$
<\omega_N\times \omega_N, [N]\times [N]>=(-1)^{\vert \omega_N\vert\vert [N]\vert} <\omega_N,[N]><\omega_N,[N]>
=(-1)^n.
$$
Therefore, by Proposition~\ref{shriek and Poincare duality}, the diagram of part 1) commutes up to the sign
$(-1)^{2n+n}(-1)^n=+1$.
The diagrams of part 2) and (3) commutes up to the sign
$(-1)^{n(2n-n)}(-1)^n=+1$.
\end{proof}

\section{Appendix: Signs and degree shifting of products}

Let $A$ be a graded vector space equipped with a morphism
$\mu_A:A\otimes A\rightarrow A$ of degree $\vert\mu_A\vert$.
Let $B$ be another graded vector space equipped with a morphism
$\mu_B:B\otimes B\rightarrow B$ of degree $\vert\mu_B\vert$.

\begin{defn}\label{definition associative and commutative}
The multiplication $\mu_A$ is {\it associative} if $\mu_A\circ(\mu_A\otimes 1)=(-1)^{
  \vert\mu_A\vert}\mu_A\circ(1\otimes\mu_A)$.
The multiplication $\mu_A$ is {\it commutative} if $\mu_A(a\otimes b)=(-1)^{
  \vert\mu_A\vert+\vert a\vert\vert b\vert}\mu_A(b\otimes a)$ for all
$a$, $b\in A$.
A linear map $f:A\rightarrow B$ is a {\it morphism
of algebras of degree $\vert f\vert$ }if
$f\circ\mu_A=(-1)^{\vert f\vert\vert\mu_A\vert}\mu_B\circ(f\otimes f)$
(In particular $\vert\mu_A\vert=\vert\mu_B\vert+\vert f\vert$).
\end{defn}
\begin{prop}
{\em (i)} The composite $g\circ f$ of two morphisms of algebras $f$ and $g$ of degrees
$\vert f\vert$ and $\vert g\vert$ is a morphism of algebras of degree
$\vert f\vert+\vert g\vert$. The inverse $f^{-1}$ of an isomorphism $f$ of algebras of degree
$\vert f\vert$ is a morphism of algebras of degree
$-\vert f\vert$.

{\em (ii)} Let $f:A\rightarrow B$ be an isomorphism of algebras of degree
$\vert f\vert$. Then $\mu_A$ is commutative if and only if $\mu_B$ is
commutative. And $\mu_A$ is associative if and only if $\mu_B$ is
associative.
\end{prop}
\begin{ex}\label{lem:important_degree}
(i) (Compare with~\cite[Remark 3.6
  and proof of Proposition 3.5]{Tamanoi:cap products})
Let $A$ be a lower graded vector space equipped with a morphism 
$\mu_A:A\otimes A\rightarrow A$ of lower degree $-d$ associative and
commutative in the sense of definition~\ref{definition associative and
  commutative}. Denote by ${\mathbb A}=s^{-d}A$ the
$d$-desuspension~\cite[p. 41]{F-H-T}
of $A$: ${\mathbb A}_i=A_{i+d}$.
Let $\mu_{\mathbb{A}}:{\mathbb{A}}\otimes {\mathbb{A}}\rightarrow
{\mathbb{A}}$ the morphism of degree $0$
given by $\mu_{\mathbb{A}}(a\otimes b)=(-1)^{d(d+p)}\mu_A(a\otimes b)$
for $a\in A_p$ and $b\in A_q$.
Then the map $s^{-d}:A\rightarrow \mathbb{A}$, $a\mapsto a$,
is an isomorphism of algebras of lower degree $-d$ and $\mu_{\mathbb A}$ is
commutative and associative in the usual graded sense.

(ii) Let $f:A\rightarrow B$ be a morphism of algebras of degree $\vert f\vert$
in the sense of definition~\ref{definition associative and commutative} with respect to the multiplications $\mu_A$ and $\mu_B$.
Then the composite
$\mathbb{A}\buildrel{(s^{\vert\mu_A\vert})^{-1}}\over\rightarrow A
\buildrel{f}\over\rightarrow B\buildrel{s^{\vert\mu_B\vert}}\over\rightarrow \mathbb{B}
$ is a morphism of algebras of degree $0$ with respect to the
multiplications $\mu_{\mathbb{A}}$ and $\mu_{\mathbb{B}}$.
\end{ex}

The following proposition explains that
the generalized cup product (Definition~\ref{cup product Hochschild}) is natural with respect to morphism of
algebras of any degree
(Definition~\ref{definition associative and commutative}).
\begin{prop}\label{naturality generalized cup product}
Let $A$ be an algebra. Let $M$ and $N$ be two $A$-bimodules.
Let $\bar{\mu}_M\in\text{\em Ext}_{A\otimes A^{op}}^*(M\otimes_A
M,M)$ and 
$\bar{\mu}_N\in\text{\em Ext}_{A\otimes A^{op}}^*(N\otimes_A
N,N)$.
Let $f\in\text{\em Ext}_{A\otimes A^{op}}^*(M,N)$ such that in the derived
category of $A$-bimodules
\begin{equation}
f\circ \bar{\mu}_M=(-1)^{\vert f\vert\vert
  \bar{\mu}_M\vert}\bar{\mu}_N\circ (f\otimes_A f)
\end{equation}
Then $HH^*(A,f):HH^*(A,M)\rightarrow HH^*(A,N)$ is a morphism of
algebras of degree $\vert f\vert$.
\end{prop}
\begin{proof}
Consider the diagram
$$
\xymatrix@C35pt@R15pt{
HH^*(A,M)^{\otimes 2}\ar[r]^-{\otimes_A}\ar[d]_-{HH*(A,f)^{\otimes 2}}
&  HH^*(A,M\otimes_A
M)\ar[r]^-{HH^*(A,\bar{\mu}_M)}\ar[d]^-{HH*(A,f\otimes_A f)}
& HH^*(A,M)\ar[d]^{HH*(A,f)}\\
HH^*(A,N)^{\otimes 2}\ar[r]_-{\otimes_A}
&  HH^*(A,N\otimes_A N)\ar[r]_-{HH^*(A,\bar{\mu}_N)}
& HH^*(A,N)
}
$$
The left square commutes exactly since for $g$,
$h\in\text{Hom}_{A\otimes A^{op}}(B(A,A,A),M)$,
$$
(f\otimes_A f)\circ (g\otimes_A h)\circ c=(-1)^{\vert f\vert\vert g\vert}
\left((f\circ g)\otimes_A (f\circ h)\right)\circ c.
$$
By equation~(1), the right square commutes up to the sign $(-1)^{\vert f\vert\vert  \bar{\mu}_M\vert}$.
\end{proof}
\begin{ex}\label{examples naturality cup product}
(i) Let $A$ be an algebra. Let $M$ be an $A$-bimodule.
Let $\bar{\mu}_M\in\text{Ext}_{A\otimes A^{op}}^d(M\otimes_A
M,M)$. Denote by $\mathbb{M}:=s^{-d}M$ the $d$-desuspension of the
$A$-bimodule $M$: for $a$, $b\in A$ and $m\in M$,
$a(s^{-d}m)b:=(-1)^{d\vert a\vert}s^{-d}(amb)$~\cite[X.(8.4)]{MacLanehomology}.
Then the map $s^{-d}:M\rightarrow \mathbb{M}$ is an isomorphism of
$A$-bimodules of degree $d$.
Consider $\bar{\mu}_{\mathbb{M}}\in\text{Ext}_{A\otimes
  A^{op}}^0(\mathbb{M}\otimes_A \mathbb{M},\mathbb{M})$
such that in the derived category of $A$-bimodules,
$s^{-d}\circ \bar{\mu}_M=(-1)^{d\vert
  \bar{\mu}_M\vert}\bar{\mu}_{\mathbb{M}}\circ ( s^{-d}\otimes_A
s^{-d})$.
Then $HH^*(A,s^{-d}):HH^*(A,M)\rightarrow HH^*(A,\mathbb{M})$ is a morphism of
algebras of lower degree $d$.
In particular, by example~\ref{lem:important_degree} (ii),
the composite
$$
\mathbb{HH^*(A,M)}:=s^{-d}HH^*(A,M)\buildrel{s^{d}}\over\rightarrow
HH^*(A,M)
\buildrel{HH^*(A,s^{-d})}\over\rightarrow
HH^*(A,\mathbb{M})
$$
is an isomorphism of algebras of degree $0$.

(ii) Let $A$ be a commutative algebra. Let $M$ and $N$ be two $A$-modules.
Let $\mu_M\in\text{Ext}_{A^{\otimes 4}}^*(M\otimes
M,M)$ and 
$\mu_N\in\text{Ext}_{A^{\otimes 4}}^*(N\otimes
N,N)$.
Let $f\in\text{Ext}_{A}^*(M,N)$ such that in the derived
category of $A^{\otimes 4}$-modules,
$f\circ \mu_M=(-1)^{\vert f\vert\vert
  \mu_M\vert}\mu_N\circ (f\otimes f)$.  From Lemma~\ref{lem:generalizeddecompositioncup}(3) and
Proposition~\ref{naturality generalized cup product},
since
$f\circ \bar{\mu}_M\circ q=(-1)^{\vert f\vert\vert
  \bar{\mu}_M\vert}\bar{\mu}_N\circ (f\otimes_A f)\circ q$,
$HH^*(A,f):HH^*(A,M)\rightarrow HH^*(A,N)$ is a morphism of
algebras of degree $\vert f\vert$.
\end{ex}

\begin{lem}\label{lem:dualsign}
{\em (i)} For a commutative diagram of graded ${\mathbb K}$-modules 
$$
\xymatrix@C20pt@R15pt{
A \ar[r]^-{f} \ar[d]_-{h} & B \ar[d]^-{k}\\
C \ar[r]_-{g} & D, 
}
$$
the square 
$$
\xymatrix@C65pt@R15pt{
A^{\vee} & B^{\vee} \ar[l]_-{f^{\vee}}\\
C^{\vee} \ar[u]^-{h^{\vee}} & D^{\vee} \ar[u]_-{k^{\vee}} \ar[l]^-{(-1)^{|f||k|+|g||h|}g^{\vee}} 
}
$$
is commutative. \\
{\em (ii)} Let $f:A\to B$ and $g:C\to D$ be maps of graded ${\mathbb K}$-modules. Then, the square is commutative:
$$
\xymatrix@C20pt@R15pt{
A^{\vee}\otimes C^{\vee} \ar[r]^-{f^{\vee}\otimes g^{\vee}} \ar[d]_-{\cong} & B^{\vee}\otimes D^{\vee} \ar[d]^-{\cong}\\
(A\otimes C)^{\vee} \ar[r]_-{(f\otimes g)^{\vee}} & (B\otimes D)^{\vee}.
}
$$
\end{lem}

\begin{proof}
(i) By definition~\cite[0.1 (7)]{Tanre}, $f^\vee(\varphi)=(-1)^{\vert\varphi\vert\vert
  f\vert}\varphi\circ f$. Therefore $(g\circ h)^\vee=(-1)^{\vert g\vert\vert
  h\vert}h^\vee\circ g^\vee$.
\end{proof}
\medskip
\noindent
{\it Acknowledgments}  
The first author thanks Jean-Claude Thomas for a precious comment which helps him 
to understand Theorem \ref{thm:main_F-T}.

\end{document}